\documentclass[reqno]{amsart}
\usepackage{latexsym}
\usepackage[english]{babel}
\usepackage{xspace}
\usepackage[bookmarksnumbered,colorlinks]{hyperref}
\usepackage{graphics}

\newcommand{\bt}{\begin{teo}}
\newcommand{\et}{\end{teo}}
\newcommand{\bco}{\begin{cor}}
\newcommand{\eco}{\end{cor}}
\newcommand{\bd}{\begin{defn}}
\newcommand{\ed}{\end{defn}}
\newcommand{\bl}{\begin{lem}}
\newcommand{\el}{\end{lem}}
\newcommand{\bpr}{\begin{prop}}
\newcommand{\epr}{\end{prop}}
\newcommand{\bere}{\begin{remark}}
\newcommand{\ere}{\end{remark}}
\newcommand{\beq}{\begin{equation}}
\newcommand{\eeq}{\end{equation}}
\def\bal#1\eal{\begin{align}#1\end{align}}
\def\baln#1\ealn{\begin{align*}#1\end{align*}}
\def\bml#1\eml{\begin{multline}#1\end{multline}}
\def\bmln#1\emln{\begin{multline*}#1\end{multline*}}
\def\bga#1\ega{\begin{gather}#1\end{gather}}
\def\bgan#1\egan{\begin{gather*}#1\end{gather*}}
\newcommand{\de}{\mathrm{d}}

\newcommand{\N}{\ensuremath{\mathbb{N}}\xspace}
\newcommand{\R}{\ensuremath{\mathbb{R}}\xspace}
\newcommand{\eps}{\varepsilon}

\newcommand{\inte}{\int_a^b\!\!}

\newcommand{\ro}{r_0} 
\newcommand{\roo}{{r_0}} 

\newtheorem{teo}{Theorem}[section]
\newtheorem{lem}[teo]{Lemma}
\newtheorem{defn}[teo]{Definition}
\newtheorem{cor}[teo]{Corollary}
\newtheorem{prop}[teo]{Proposition}
\theoremstyle{remark}
\newtheorem{remark}[teo]{Remark}
\theoremstyle{definition}
\newtheorem{example}[teo]{Example}

\newcommand{\be}{\begin{equation}}
\newcommand{\ee}{\end{equation}}

\newcommand{\gam}{\gamma}

\title[Convex regions of stationary spacetimes]{Convex regions of stationary
spacetimes and Randers spaces.  Applications to lensing and asymptotic flatness}

\author[E. Caponio]{Erasmo Caponio}
\address{Dipartimento di Meccanica, Matematica e Managment\hfill\break\indent
Politecnico di Bari \hfill\break\indent Via Orabona 4,
70125, Bari, Italy}
\email{caponio@poliba.it}

\author[A.V. Germinario]{Anna Valeria Germinario}
\address{Dipartimento  di Matematica\hfill\break\indent
Universit\`a degli Studi di Bari\hfill\break\indent Via  Orabona 4,
70125 Bari Italy}
\email{germinar@dm.uniba.it}
\author[M. S\'anchez]{Miguel S\'anchez}
\address{Departamento de Geometr\'{\i}a y Topolog\'{\i}a \hfill\break\indent
 Facultad de Ciencias, Universidad de Granada \hfill\break\indent
 Campus Fuentenueva s/n, 18071 Granada, Spain}
\email{sanchezm@ugr.es}
\thanks{EC and AVG are partially supported by PRIN2009 ``Metodi variazionali ed applicazioni allo
studio di equazioni differenziali nonlineari'';   \\
\indent MS is partially supported by
Spanish MTM2010-18099 (MICINN) and P09-FQM-4496 (Junta de
Andaluc\'{\i}a) grants, both with FEDER funds. }
\thanks{This research is a result of the
activity developed within the Spanish-Italian Acci\'on Integrada HI2008.0106/Azione Integrata
Italia-Spagna IT09L719F1.}
\date{}

\begin{document}
\begin{abstract}
By using {\em stationary-to-Randers correspondence} (SRC, see
\cite{cys}), a characterization of  light and time-convexity  of
the boundary of a region of a standard stationary
$(n+1)$-spacetime  is obtained, in terms of the convexity of the
boundary of a domain in a  Finsler $n$ or $(n+1)$-space  of
Randers type. The latter  convexity  is   analyzed in depth
and,  as a consequence, the  causal simplicity and the existence
of  causal  geodesics confined in the region and connecting a
point to a stationary line are characterized.  Applications to
asymptotically flat spacetimes include  the light-convexity of
 hypersurfaces $S^{n-1}(r)\times\R$, where $S^{n-1}(r)$ is a sphere of large radius in a spacelike section of 
an end,   as well as the characterization of  their 
time-convexity with natural physical interpretations. The {\em
lens effect}  of both light rays and freely falling massive
particles with a finite lifetime, (i.e. the multiplicity of such
connecting curves) is characterized in terms of the focalization
of the geodesics in the underlying  Randers manifolds. 

\end{abstract}

\maketitle

\newpage
\tableofcontents
\newpage

\section{Introduction}

Assume that a stellar object emitted  radiation (light rays or
massive particles, the latter possibly with a finite lifetime) in
the past, and we can assume that the spacetime is  stationary on a
region which includes that source and ourselves. When can we
ensure that we will receive such a radiation and, in this case,
when will it be focalized towards us, obtaining so multiple images
of the source? The key ingredients will be the existence of a {\em
light or time convex hypersurface} around the region, plus the
possibility of lensing effects due to either curvature or
topology.

{\em Convexity} is a basic property of subsets of Euclidean space
which admits several extensions to semi-Riemannian manifolds. For
a  domain $D$ of a complete Riemannian manifold, its convexity
will mean  geodesic connectedness by means of {\em minimizing}
geodesics in $D$. This is equivalent to the {\em local} and {\em
infinitesimal} convexity of its boundary $\partial D$, when
$\partial D$ is regular enough (see  complete details below). For
a Lorentzian manifold $(M,g)$, these notions for $\partial D$
admit a straightforward analog, even extensible to causality types
(time, space and lightlike  convexities). Nevertheless, the
notion of convexity for $D$ is not so clear, as just  geodesic
connectedness makes sense in general, but local {\em extremal}
properties appear only for geodesics of causal type.  However, for
(conformally) stationary spacetimes, the recent progress in its
causal structure \cite{cys} (which is related with elements of
Finsler Geometry, \cite{cym,cys}), plus the developments on
convexity in the Finslerian
 setting \cite{bcgs},  allow  to carry out a detailed explanation of
 convexity  in this Lorentzian setting, to be studied
 here. 

Recall that a spacetime is called {\em stationary} if it admits a
timelike Killing vector field $Y$. Locally, each stationary
spacetime $(L , g_L) $ admits  a {\em standard form} (namely,
$L=S\times \R$, $g_L=g_0+2\omega_0 \de t-\beta\de t^2$, for some
{\em lapse} $\beta$ and {\em shift} $\omega_0$, see below) and,
when this form can be obtained globally, the spacetime is called
{\em standard stationary}. This structure will be assumed here,
and it is not too restrictive, as the completeness of $Y$ plus the
property of being distinguishing for $L$ (a causality condition
less restrictive than strong causality) ensure it \cite{js}.
Standard stationarity becomes natural in the framework of
asymptotically flat spacetimes,  including black holes, as the
``no hair'' results postulate that these will stabilize in some
member of the  Kerr family. Moreover, as  our techniques will be
conformally invariant in some cases, the corresponding results
will be extensible  to spacetimes such as the classical FLRW
ones.

Even though connections with Finsler Geometry were pointed out
long time ago \cite{Li1, Li2},  results about causality of
standard stationary spacetimes obtained by using an accurate
relation with such a geometry, have been obtained only recently
(see \cite{cym, cys}).  This relation is based on the fact that
the projections on $S$ of lightlike geodesics in $(L , g_L) $ are
pregeodesics for a suitable Finsler metric on $S$ of Randers type,
called {\em Fermat metric} in \cite{cym}. Such a simple property
leads to a kaleidoscope of relations between the Causality of
standard stationary spacetimes and the geometry of Randers spaces,
or {\em stationary-to-Randers correspondence (SRC)}, carefully
developed in \cite{cys}.

In this paper, our aim is to use SRC in order to describe the
convexity of a stationary region $D\times \R$ of a standard
stationary spacetime $L=S\times \R$. The light-convexity and
time-convexity of $D\times \R$ are characterized in terms of the
convexity  w.r.t. the geodesics of a  Randers metric on,  resp., $D$ and the product
 $\R_u\times D$, where $\R_u\equiv \R$ (the subindex $u$ will
 be used throughout the paper to recall the natural coordinate on the factor
 $(\R , du^2)$, not to be confused with
the global time coordinate $t$ associated to other copy of $\R$).
Moreover, these elements are also characterized in terms of
the existence of a causal geodesic with fixed length $l\geq 0$ 
(which is intended not to be bigger than the mean lifetime of the
traveling particle)  connecting a point and an integral line
of the Killing field $Y$ and {\em minimizing the arrival time}  $t$. 
These results complement the classical Fermat principle
(\cite{Kov,Per}) which assures that, in the case that a time
minimizing connecting lightlike curve exists, then it must be a
lightlike geodesic 
--but it does not assure existence.  We emphasize that our
results are not merely sufficient conditions to ensure that the
connecting causal geodesic will exist. On the contrary, the Fermat
metric provides the geometric framework to fully characterize
their existence  in both cases, lightlike geodesics and timelike
geodesics with a prescribed length.

The paper is organized as follows. In Section \ref{s2}, after a
summary on the notion of convexity for a Finsler manifold
(including the recent progress in \cite{bcgs}), the convexity of
the domain $D$ for a Randers metric $R=\sqrt{h}+\omega$ is
characterized (Proposition \ref{peq}) and discussed (Examples
\ref{r1}, \ref{r2}). As a consequence, the convexity of large
balls in  asymptotically flat Randers spaces  is shown
(Proposition \ref{largesphere}).   For this result, only 
the decay of $d\omega$ (rather than $\omega$) becomes relevant
(formula \eqref{asymptflat}).

In Section \ref{s3}, the convexity of domains $D\times \R$  in a
standard stationary spacetime $L=S\times\R$  is characterized. For
light-convexity, the infinitesimal convexity of $\partial (D\times
\R)$ becomes equivalent to the infinitesimal convexity of
$\partial D$ with respect to the Fermat metric $F$ (Theorem
\ref{lighteq}). Then, SRC and the results in \cite{bcgs} yield
easily the equivalence with the local notion of light-convexity
(Corollary \ref{lightequiv}).
For time-convexity, a further insight is obtained by
using the fact that timelike geodesics can be obtained as
projections of  lightlike geodesics
 of a product manifold $\R_u\times L$ of one dimension more.  The equivalence between
the infinitesimal time-convexity of the boundary
$\partial(D\times\R)$ and the infinitesimal convexity of
$\R_u\times  \partial D$ for a suitable Randers metric $F_\beta$
is detailed (Theorem \ref{peq2}).  Moreover, the infinitesimal
convexity of the latter hypersurface is characterized
(Proposition~\ref{Rbetaconvexity2}).

These results are applied to asymptotically flat stationary
spacetimes in  Section \ref{appflat}. Concretely, a
notion of asymptotic flatness  (Definition~\ref{referee2def})
specially adapted to this setting, is introduced and discussed
along Subsection~\ref{KN}. In the
next subsection, the light-convexity of  the hypersurfaces $S^{n-1}(r)\times\R$ for $r$ 
sufficiently large 
is proven (Corollary \ref{NewCorollary1}), and the
hypotheses under which time-convexity  holds  (or is violated) are
provided (Corollary \ref{NewCorollary3}). Remarkably,
time-convexity will not hold for large  spheres $S^{n-1}(r)$ under general
physical assumptions (Proposition \ref{positivemass}),  in
contrast with the lightlike case. In Subsection \ref{kerrshell},
the paradigmatic case of Kerr spacetime is analyzed specifically.
The non time-convexity of  $S^{n-1}  (r) \times\R$, for all large enough $r$, 
is interpreted (Remark
\ref{fall}), and the  hypersurfaces  close to the stationary limit
 one are  also taken into account (Corollary \ref{cKerrlc}).

In the first subsection of Section \ref{s4},  a full
characterization of the problem of connecting a point $(p,t_p)$
and a stationary line $l_q=\{(q,t):t\in \R \}$ by means of a
(first-arriving)  future-pointing  lightlike geodesic  contained in a stationary
domain $D\times\R$  is obtained (Theorem \ref{appl1}). This is
characterized alternatively in terms of:  (a)
Geometric/Variational interpretations of  the boundary:
light-convexity of $\partial D\times \R$, (b)Finsler geometry:
convexity of $D$ with respect to the associated Fermat metric, and
(c) Causal structure: causal simplicity  of the domain $D\times
\R$.
 When $D$ is not contractible, infinitely many connecting
lightlike geodesics (with diverging arrival times) appear. 
This can be interpreted as a {\em topological  lens  effect}
(while the gravitational lensing   depends strictly on the
curvature of the Fermat metric).  As emphasized in Remark
\ref{rrr}, these conclusions and the usage of SRC here, complete
the circle of results and techniques in papers on boundaries such
as \cite{fgm, gm}, where variational methods are applied to the
study of Lorentzian geodesics\footnote{Recall that these
techniques were initiated in \cite{bfg} with the introduction of
time and light convexity in the  static case. This case becomes
quite simpler, as it is related to Riemannian instead of properly
Finslerian metrics, see \cite{bgs1}. The techniques also apply to
periodic trajectories and other Lorentzian variational problems on
convex domains (see the subtleties in \cite{bs} and references
therein).}. Finally, in Remark \ref{referee2rem2}, further
physical applicability of the results is pointed out.

In Subsection  \ref{s4.2},  previous results are extended
to timelike geodesics.  From the technical viewpoint, the
following difficulty is worth pointing out. Our main result
(Theorem \ref{appl2}) is proved by using and auxiliary product
spacetime $\R_u \times L$. Nevertheless, our hypotheses are posed
naturally on the original stationary domain $D$, rather than
on the auxiliary elements in the product spacetime. For the
connection between the hypotheses on these two spacetimes (see
Lemma \ref{lcompact}), a small improvement on the results of
convexity for Finslerian metrics is carried out (Remark
\ref{rapostilla}).  By using this method, Theorem \ref{appl2}
assures  the existence of a connecting  future-pointing  timelike geodesic with
{\em a priori} fixed Lorentzian length and minimizing the arrival
time $t$ at the stationary curve $l_q$. Removing the minimizing
property, when $D$ is not contractible one obtains also the
multiplicity of such connecting timelike geodesics. That is, any
freely falling massive particle, starting at some event $p$, will
be able to reach $l_q$ and, if $D$ is not contractible, 
arriving after unbounded values of time $t$, {\em even if the
lifetime of the particle is arbitrarily small.} 

Due to the technical subtleties of our approach, in Subsection
\ref{s4.3} a revision of the available causal, topological and
variational tools for this kind of problems, is carried out. We
stress how stationary-to-Randers Correspondence fits with the
other techniques to provide a complete solution of causal geodesic
connectedness in the stationary setting, and point out further
related  problems.

Finally, in the last section the conclusions are summarized.

\section{Convexity of domains of Randers manifolds} \label{s2}
\subsection{Finsler metrics}
Let us recall some notions about Finsler manifolds. A Finsler
structure  on a smooth finite dimensional (connected) manifold $M$
is a function $F\colon TM\to[0,+\infty)$ which is continuous on
$TM$, smooth on $TM\setminus 0$, vanishing only on the zero
section, fiberwise positively homogeneous of degree one (i.e.
 $F(\lambda y)=\lambda F(y)$, for all $y\in TM$
and $\lambda>0$), and which has fiberwise strongly convex
square, that is, the matrix
\beq\label{fundmatr}
 g_{y} =\left[\frac{1}{2}\frac{\partial^2 (F^2)}{\partial
y^i\partial y^j} (y) \right]\eeq
is   positive definite   for any
$ y \in TM\setminus 0$.  Observe that $y\in TM\setminus 0\mapsto g_y$ is a symmetric section of the tensor product of the pulled back   cotangent bundle $\pi^*T^*\!M$  over $TM\setminus 0$ with itself. Henceforth, besides to the quite standard notation $A(y)$,   we will  also use -- to get more compact formulas -- an index that indicates the dependence on $y\in TM\setminus 0$ for  sections $A$ of $\pi^*TM$,  its dual $\pi^*T^*\!M$ or their tensor product (for example,   $g_y$ above or $(H_\phi)_y$   for the Hessian, with respect to the Chern connection, of a function $\phi$ on $M$).

The minimal requirement about the
regularity of $F$ that we need is that the fundamental tensor $g$
is
$C^{1,1}_{\rm loc}$ in $TM\setminus 0$.

By homogeneity, $F(y)=g_y(y,y)$, for all $y\in TM$, thus
the fundamental tensor $g$  gives the shape of the unit sphere (indicatrix),   $F(y)=1$, $y\in T_x
M$,   at each point    $x\in M$.
Hence,   its positive definiteness yields the  strict  convexity of the closed ball $\bar
B_x=\{y\in T_x M:F(y)\leq 1\}$,  that is   any line segment joining two points
contained in $\bar B_x$ is contained in $B_x$, except, at most, its endpoints    (see,
for example, \cite[Exercise 2.1.6]{bcs}).

The length of a piecewise smooth  curve
$\gamma\colon [a,b]\to M$ with respect to the Finsler  metric 
$F$  is defined by
\[\ell_F(\gamma)=\int_a^b\!\!
F (\dot\gamma) \; \de s\]
hence the Finsler
distance between two arbitrary points $p, q\in M$ is
given by
\[
d (p,q)= \inf_{\gamma\in {\mathcal P}(p,q; M)}{\ell}_F(\gamma),
\] where ${\mathcal P}(p,q;M)$ is the set of all piecewise smooth
curves $\gamma\colon[a,b]\to M$ with $\gamma(a)=p$ and
$\gamma(b)=q$.  The distance function is non-negative and
satisfies the triangle inequality, but  in general it is not
symmetric  since $F$ is only {\em positively} homogeneous in $y$,
that is, $d$ is a {\em generalized distance} (see \cite{FHS} for
an exhaustive study). As a consequence, the {\em reverse} Finsler
metric of $F$ is defined as $\tilde F (y) =F (-y) $. So, for any
point $p\in M$ and for all $r>0$, we can define two different
balls centered at $p$ and having radius $r$: the {\em forward
ball} $B^+(p,r)=\{q\in M \mid d (p,q)<r\}$ and the {\em backward}
one $B^-(p,r)=\{q\in M \mid d (q,p)<r\}$. Analogously, it makes
sense to introduce two different types of Cauchy sequences and
completeness: a sequence $(x_n)_n\subset M$ is a {\em forward}
(resp. {\em backward}) {\em Cauchy sequence} if for all $\eps>0$
there exists an index $\nu\in\N$ such that for all $m\geq n\geq
\nu$, it is $d(x_n,x_m)<\eps$ (resp. $d (x_m,x_n)<\eps$);
consistently a Finsler manifold is {\em forward complete} (resp.
{\em backward complete}) if every forward (resp. backward) Cauchy
sequence converges\footnote{It is worth pointing out that a second
natural notion of forward and backward Cauchy sequence can be
given, see \cite[Section 3.2.2]{FHS}. This notion is {\em not}
equivalent to  that  stated above, but it yields equivalent
(forward and backward) Cauchy completions and, thus, equivalent
notions of completeness.}. In general, the backward elements for
$F$ are forward for $\tilde F$, and we will refer just to forward
elements. It is well known that the topology generated by the
forward balls coincides with the underlying manifold topology.
Moreover, an adapted version of the Hopf-Rinow theorem  holds (cf.
\cite[Theorem 6.6.1]{bcs}) stating, in particular, the equivalence
between forward completeness and compactness of closed and forward
bounded subsets (i.e., those included in some forward ball) of
$M$.
\subsection{Convexity}
We say that a Finsler manifold $(M,F)$ is {\em (geodesically)
convex} if each pair of points $(p,q)\in M\times M$ can be
connected by a (non-necessarily unique) minimizing geodesic, i.e. a
geodesic with length $d(p,q)$, starting at $p$ and ending at $q$.
Recall that convexity for $F$ is equivalent to convexity for
$\tilde F$. Any of the assumptions of  the  Hopf-Rinow theorem (namely,
either forward or backward completeness) imply convexity. However,
after \cite{cys}, it becomes clear that the forward or the
backward completeness of the generalized metric $d$ can be
substituted in several classical results (as convexity,
Bonnet-Myers or Synge theorems) by the assumption of the
compactness of the closed balls with respect to the symmetrized
distance $d_s$  associated to $d$, namely
\[
d_s(p,q)=\frac 12 \left(d(p,q) + d(q,p)\right), \quad \quad
\forall p,q\in M.
\]
More precisely, let $B_s$ denote the balls with respect to $d_s$.
If  the closed balls $\bar B_s(x,r)$ are compact for all $x\in M$
and $r>0$ (or equivalently the subsets  $\bar B^+(x,r_1)\cap \bar
B^-(y,r_2)$ are compact for any $x,y\in M, r_1, r_2>0$), then
$(M,F)$ is convex, \cite[Theorem 5.2]{cys}. It is worth to stress
that the Hopf-Rinow theorem does not hold in general for the
metric $d_s$. For instance, Example  2.3 in \cite{cys} exhibits a
non compact, $d_s$-bounded Randers space whose   symmetrized
distance $d_s$ is complete.

From a variational viewpoint, geodesics parametrized with constant
speed (i.e $s\mapsto F(\gamma(s),\dot\gamma(s))={\rm const.}$) and
connecting two fixed points $p$ and $q$ on  $(M,F)$, are the
critical points of the energy functional
\[J(\gamma)=\frac 12\int_a^b F^2 (\dot \gamma) \;
\de s \]
defined on the manifold of  the  $H^1$ curves $\gamma$ on $M$, parametrized on the
interval $[a,b]\subset\R$ and such that $\gamma(a)=p$, $\gamma(b)=q$
(see, for example, \cite[Proposition 2.1]{cym}).

The convexity of a domain  (i.e., an  open connected subset)
$D\subset M$, regarded as a Finsler manifold in its own right, can
be related to the {\em infinitesimal convexity} of its boundary
$\partial D$, at least when the closure $\bar D$ is  a manifold
with boundary and $\partial D$ is   at least twice   continuously differentiable,
i.e. when $\partial D$ is (locally and then globally) the inverse
image of a regular value of some $C^r$ function  with $r\geq 2$.
Infinitesimal
convexity means that for each $x\in\partial D$ there exists a
neighborhood $U\subset M$ of $x$ such that  for one (and then
for all) $C^2$ function $\phi: U \rightarrow\R$ such that
\be\label{defb1}
\begin{cases}
          \phi^{-1}(0) =  U\cap\partial D \\
          \phi > 0&\mbox{on $U\cap D$ }\\
          {\rm d}\phi(x)\not= 0&\text{for every $x\in U\cap \partial D$}\end{cases}
\ee one has \be\label{defb2} (H_\phi)_y(y,y)\leq 0  \quad
\text{for every $y\in T_x\partial D\setminus\{0\}$,} \ee  where
$H_\phi$ is the Hessian of $\phi$ with respect to the Chern
connection $\nabla$ of $(M,F)$, i.e. $H_\phi=\nabla(\de \phi)$   (see \cite[Section 2.4]{bcs}).
More precisely, in natural coordinates on
$TM\setminus\{0\}$,  $ (H_\phi)_y(u,v)$  is given by
(the Einstein summation convention is used in the remainder)
\[\big((H_\phi)_y\big)_{ij}u^iv^j=\frac{\partial^2\phi}{\partial x^i\partial
x^j} u^iv^j-\frac{\partial \phi}{\partial x^k} \Gamma^k_{\,\,ij}(y)u^iv^j,\]
where $\Gamma^k_{\,\,ij}   = \Gamma^k_{\,\,ij}(y)   $ are the components of the Chern
connection given by
\[ \Gamma^k_{\,\,ij}= \frac{g^{ks}}{2} \left( \frac{    \delta   g_{si}}{  \delta   x^j}
-  \frac{   \delta   g_{ij}}{   \delta   x^s}
+ \frac{   \delta  g_{js}}{  \delta   x^i}
\right). \]
 Here,   $g_{ks}=g_{ks}(y)$ and $g^{ks}=g^{ks}(y)$ are, respectively,  the components of $g_y$  (see
\eqref{fundmatr}) and of its inverse at $y\in TM\setminus 0$, while
$\frac{ \delta }{\delta x^i}$ are vector fields on $TM$ defined  as
\[\frac{ \delta }{\delta x^i}=\frac{\partial}{\partial x^i}-N^j_i(y)\frac{\partial}{\partial y^j},\]
where $N^j_i=N^j_i(y)$ are the components of the so-called {\em non-linear connection on $TM\setminus 0$} (see \cite[\S 2.3]{bcs}).
As the equation of a geodesic $\gamma=\gamma(s)$, parametrized with constant speed,  i.e. $F(\dot\gamma)=\mathrm{const.}$,  is given by
\[
\frac{\de ^2 \gamma^i}{\de s^2}+\Gamma^i_{\,\,jk}(\dot \gamma)\dot\gamma^j\dot\gamma^k=0,\]
it is immediate  to see that
\beq(\phi\circ\gamma)''(s)=(H_\phi)_{\dot\gamma(s)}(\dot\gamma(s),\dot\gamma(s))\label{evH}.\eeq
\bere\label{manyconvex} Eq. \eqref{evH}  holds also for a
Riemannian or a Lorentzian metric $g$, namely if $\gamma=\gam(s)$
is a geodesic of $g$ then
$(\phi\circ\gamma)''(s)=H_\phi(\dot\gamma,\dot\gamma)$, where, in
this case,  the Hessian of $\phi$ is $H_\phi=\nabla(\de \phi)$ and
$\nabla$ is the Levi-Civita connection of $g$. \ere

\bere\label{local convexity}  This notion of infinitesimal
convexity  for a hypersurface is a natural extension  to the
Finslerian setting of the analogous  one in a Riemannian manifold. 
Let us summarize
the relation between this notion  and the ones  of {\em local} and {\em strong} convexity for $\partial D$. 
Recall that, on one hand,  an embedded  hypersurface $N$ is {\em locally convex} when
for each $x\in N$ a small enough neighborhood $U$ of $x$ exists
such that all the geodesics in $U$ issuing from $x$ and tangent to
$N$ lie in the closure of one
 of the two connected parts of $U\backslash N$, called the {\em local
exterior}.  Recall that, as the hypersurface $\partial D$  in
\eqref{defb1} is the boundary of a domain,      all the locally
defined functions $\phi$ can be taken so that they match in a
global one, and the global exterior (namely, $M\backslash D$) is
well defined. On the other hand, when \eqref{defb2} is satisfied
with the strict inequality, we will say that $\partial D$ is
{\em   strongly convex} at $x\in\partial D$.

 Trivially, strong convexity at a point implies both, local and infinitesimal convexity on a neighborhood of that point, and it is also clear that the local convexity at a point implies the infinitesimal one at the same point (cf. e.g.
\cite[Prop. 14.2.1 and Th. 14.2.3]{sh2}). The non-triviality of the last converse when the inequality \eqref{defb2} is not strict   (for that, one
needs also to assume that the inequality holds  on a neighborhood
of the point, otherwise the implication is not true),  was
stressed by Bishop \cite{b} (see also  the
review \cite{San}), who proved this equivalence in the Riemannian
case
 for a $C^4$ metric. The proof of the equivalence in
the general Finsler case   was obtained recently in
\cite[Corollary 1.2]{bcgs},  where,  the degree of differentiability
 was also lowered to $C^{1,1}_{\mathrm{loc}}$ (i.e. $C^1$ on
$TM\setminus\{0\}$ with locally Lipschitz differential) for the
fundamental tensor $g$ and $C^{2,1}_{\mathrm{loc}}$ for the
function $\phi$.
 \bere\label{c2}
We point out that  a refinement of the proof in \cite{bcgs} on  the
equivalence between infinitesimal convexity and local one allows to  optimize  the degree of differentiability
of the hypersurface $\partial D$ to $C^2$ (see \cite{caponiogelogra}).
\ere

Finally, recall also that when the subsets $\bar B_s(x,r)\cap \bar D$ are
compact for all $x\in D, r>0$, the above equivalent notions of
convexity for $\partial D$ are also equivalent to the convexity of
$D$ (see, \cite[Theorem 1.3]{bcgs}). \ere

\subsection{Randers spaces} In this paper, we deal with the convexity of a domain in  a
Randers space. Given a Riemannian manifold $(S, h)$ and a one-form
$\omega$ on $S$ such that, for any $x\in S$, $\| \omega \|_x < 1$,
where $ \| \omega \|_x = \sup_{ y \in T_x S \setminus \{0\}} |
\omega (y)  |  /  \sqrt{h(y,y)}$, a Randers metric $R$ and
its reversed one $\tilde R$ on $S$ are defined by setting
\be
\label{ran} R(y) = \sqrt{h(y,y)} + \omega(y)\quad\text{and}\quad \tilde R(y) = \sqrt{h(y,y)} -
\omega (y), \quad  y\in TS
\ee
Condition $\| \omega \|_x < 1$ is necessary and
sufficient for $R$ and $\tilde R$ to be positive and it implies
that they have fiberwise strongly convex square   (see \cite[\S
11.1]{bcs} for details).

The condition  \eqref{defb2}  at a point of $\partial D$ in
a Randers space $(S,R)$ can be  written in terms of the Hessian
$H^h$ of $\phi$ with respect to the Levi-Civita connection of the
Riemannian metric $h$ plus another term involving $ \de \omega$.

In what follows, $\nabla^h$ will    denote the gradient  symbol  with respect to the metric $h$
as well as  and the Levi--Civita connection of $h$  and
  $\widehat{\de \omega }$ the
$(1,1)$--tensor field $h$-metrically associated to $\de \omega$, i.e. for every $(x,y) \in TM$,
$\de \omega_x(\cdot, y) = h_x\big (\cdot, \widehat{ \de \omega}(y)\big)$.
\begin{prop}\label{peq}
Let $D$ be a domain of class $C^{2}$ of  a Randers  manifold
 $(S,R)$,   $x\in \partial D$,  $\phi$ be  a  function defined on a neighborhood of $x$ satisfying conditions \eqref{defb1}.
Then, the following propositions are equivalent:
\begin{itemize}
\item[(i)] $(\partial D ; R)$ is  infinitesimally convex at $x$;
\item[(ii)] $(\partial D ; \tilde R)$ is  infinitesimally convex at $x$;
\item[(iii)] for all $y \in T_x \partial D$
\be  \label{hes}
H_\phi^h (y,y) + \sqrt{h(y,y)}\,   \de \omega (y,  \nabla^h \phi) \leq 0;
\ee
\item[(iv)] for all $y \in T_x \partial D$
\beq\label{hes2}
H_\phi^h (y,y) - \sqrt{h(y,y)}\, \de \omega (y,  \nabla^h \phi )
\leq 0;
\eeq
\item[(v)]for all $y \in T_x \partial D$
\beq\label{hes3}
H_\phi^h (y,y) + \sqrt{h(y,y)}\,\big| \de \omega (y,  \nabla^h \phi)\big|
\leq 0.\eeq
\end{itemize}
\end{prop}
\begin{proof} It is enough to show that inequality \eqref{hes} is equivalent to
the inequality  \eqref{defb2}. Indeed the other equivalences will
follow simply by the definition \eqref{ran} of $R$ and $\tilde R$
 and observing  that \eqref{hes2} is the evaluation of
\eqref{hes} in $-y$. If $\gamma$ is any smooth curve on $S$, we
can compute the  second derivative of $\phi\circ\gamma$ by using
the Levi-Civita connection of $h$. As
$(\phi\circ\gamma)'=h(\nabla^h\phi, \dot\gamma)$  and  $H_\phi^h
\big(\dot \gamma (s) , \dot \gamma
(s)\big)=h\big(\nabla^h_{\dot\gamma(s)}\big(\nabla^h\phi(\gamma(s))\big),\dot\gamma(s)\big)$
(see e.g.  \cite[Ch. 3, Lemma 49]{onei}), we obtain
 \bal
(\phi\circ\gamma)''(s)&=
h\big(\nabla^h_{\dot\gamma(s)}\big(\nabla^h\phi(\gamma(s))\big),\dot\gamma(s)\big)+
 h \big(\nabla^h
\phi (\gamma (s) ), \nabla^h_{\dot \gamma (s) } \dot \gamma
(s)\big)\nonumber\\
&=H_\phi^h
\big(\dot \gamma (s) , \dot \gamma (s)\big) +
 h \big(\nabla^h
\phi (\gamma (s) ), \nabla^h_{\dot \gamma (s) } \dot \gamma
(s)\big).\label{ddotrho} \eal
A  geodesic
$\gamma=\gamma(s)$ of $(S,R)$ (parametrized with constant Randers
speed), satisfies in particular  the pregeodesic equation
(see
for example the computations   in  \cite{cym1} above its Eq.\  (6)):
\beq
\nabla^h_{\dot \gamma} \dot \gamma    =
\sqrt{h(\dot \gam, \dot \gam )} \,  \widehat{ \de \omega}( \dot \gam )
     +\frac{1}{2}\frac{\de}{\de s}
\left(\log(h(\dot\gam,\dot\gam))\right)\dot\gam.
\label{georandersh}
\eeq
Now, for any
$x \in \partial D$ and  $y \in T_x
\partial D\setminus \{0\}$ consider the geodesic
$\gamma$ such that $\gamma (0) =x$ and $\dot \gamma (0) = y$. As
$h_x (\nabla^h \phi (x ), y ) =0$, substituting
\eqref{georandersh} in \eqref{ddotrho} and recalling \eqref{evH},
we obtain
\[
(H_\phi)_y(y,y)=H_\phi^h (y,y) + \sqrt{h(y,y)}  h \big( \nabla^h \phi ,
 \widehat{ \de \omega }( y ) \big),
\]
i.e., the expression in the left-hand side of \eqref{hes}.
\end{proof}
Clearly, Proposition~\ref{peq}  can be extended to
 strong convexity.

By \eqref{hes}, the Hessians of $\phi$ for $R$ and $h$ will agree
on the vectors tangent to $\partial D$  if and only if  $\de
\omega (\nabla^h\phi, \cdot)$  vanishes there. However,  from \eqref{hes3} the
following holds:

\bco If $(\partial D; R)$ is infinitesimally convex
then also $(\partial D;h)$ is infinitesimally convex. \eco

The following example shows that
the converse is not true.
\begin{example} \label{r1}
Let $S= \R^2$ be endowed with a Randers metric as in
\eqref{ran}, being $h$ the usual Euclidean metric and $\omega$
defined as $\omega_x(y) = f (x^2 ) y^1 $ for any $x = (x^1 , x^2
), y = (y^1 , y^2 ) \in \R^2$ , where $f : \R \rightarrow \R$
is a smooth function such that $| f | < 1$. For any $\ro
>0$, take as a domain the open ball
$$D_{\ro} = \{ (x^1 , x^2 ) \in \R^2 \mid  (x^1)^2 + (x^2)^2 <  \ro^2  \}$$
 whose boundary is defined by  $\phi (x^1 , x^2 ) =  \ro^2  - (x^1)^2 -
(x^2)^2$. Obviously, $\nabla^h \phi (x) = (-2x^1 , -2 x^2)$,
$H^h_\phi(y,y) = -2 (y^1)^2 -2 (y^2)^2$, and $\partial D_{\ro}$ is
convex with respect to $h$ for any $\ro>0$. Nevertheless, it is
easy to find cases  where  $(\partial D_{\ro} ;R)$ is not convex. For
example, it is enough to assume $\ro |f'(\ro)|>1$ (in addition to
$|f|<1$). Indeed, $\de \omega (y,z) = f' (x^2) y^1 z^2 - f' (x^2)
y^2 z^1$, and the convexity condition \eqref{hes3}   can be
written as
\begin{equation}\label{eexconv} - (y^1)^2 - (y^2)^2 +\sqrt{(y^1)^2 + (y^2)^2} |
f' (x^2) |  | x^2 y^1 - x^1 y^2 |
  \leq  0 \end{equation}
for any $x \in \partial D_{\ro}$, $y \in T_x \partial D_{\ro}$.
But \eqref{eexconv} is not fulfilled for $x=(0,\ro)$, $y=(1,0)$.

\end{example}
Next, the previous example is modified in order to show that the
convexity of $(\partial D; R)$ does not imply the convexity of
$\partial D$ with respect to the Riemannian metric $h_0=
h-\omega^2$. This question becomes  natural because, on one hand,
$\sqrt{h_0+\omega^2} + \omega$ defines always a Randers metric
(with no restriction on $\omega$), and, on the other, such an
$h_0$ becomes the Riemannian metric on the slices of a standard
stationary spacetime (see Remark~\ref{reverse}(2) below).
\begin{example} \label{r2}
Redefine, in Example \ref{r1}, the 1-form as $\omega_x (y) = f (x^1
) y^1 $.  Since $\de \omega=0$, $(\partial D_{\ro}; R)$ is 
convex from Proposition \ref{peq}. Let  $h_0=h-\omega^2$,
i.e.,
$$h_{0_{(x^1,x^2)}} \big( (y^1 , y^2) , (y^1 , y^2)\big) = (1- f(x^1 )^2)(y^1)^2 + (y^2)^2.$$
To check that $(\partial D_{\ro}; h_0)$  is not convex for simple
choices of $f$, recall that  a curve $\gamma (s) = ( x^1 (s) , x^2
(s) )$, $ s \in I$, is a geodesic for $h_0$  iff
\[
\begin{cases}
\ddot x^1 = \dfrac{f (x^1)f' (x^1)}{1 - f(x^1)^2}(\dot x^1)^2   \\
\ddot x^2 = 0.
\end{cases}\]
So the $h_0$-Hessian of  $\phi$
 is \beq\label{hg} H^{h_0}_\phi  \big( (y^1 , y^2) , (y^1
, y^2)\big) = -2 \left( 1+ x^1 \frac{f (x^1)f' (x^1)}{1 -
f(x^1)^2} \right) (y^1)^2 - 2 (y^2)^2. \eeq Notice that, as $(y^1
, y^2)$ is assumed to be tangent to $\partial D_{\ro}$ at $(x^1 ,
x^2)$, necessarily $y^2=-x^1 y^1/ x^2$ whenever $x^2\neq 0$. So,
the right part of \eqref{hg} reads:
$$
-2 \left( 1+ x^1 \frac{f (x^1)f' (x^1)}{1 - f(x^1)^2} +
\frac{(x^1)^2}{\ro^2-(x^1)^2} \right) (y^1)^2 .
$$
Thus, for each $x^1\neq \pm \ro, 0$ and any choice of $f(x^1)$, we
can choose $f'(x^1)\neq 0$ so that \eqref{hg} becomes positive for
$y^1\neq 0$.
\end{example}
\subsection{Convexity in asymptotically flat
Randers manifolds}\label{asflat} The notion of asymptotic flatness
is specially relevant for Riemannian manifolds, and  it  allows to
ensure that large balls are convex. Next we explore the analogous
issues for a Randers manifold.

Consider a Riemannian manifold $(S, h)$ endowed with a one-form
$\omega$, with $\|\omega\|_x<1$, for each $x\in S$. Assume that
there is a compact set $K\subset S$ such that $S \setminus K$ is a
disjoint union of ends, $E^{(k)},k=1,\dots , m$, such that each
end is diffeomorphic to  $\R^n \setminus \{0\}$  and in each end
there exist  an {\em (asymptotic)}  coordinate chart  $x=(x^1,\ldots,x^n)$
 and positive constants  $p$ and $q$ such that $h$ and $\omega$
satisfy \bal
&h_{ij} = \delta_{ij} + O(1/|x|^p),\nonumber\\
&\partial_k h_{ij} =O(1/|x|^{p+1}),\label{asymptflat}\\
&\Omega_{ij}:=\partial_i
\omega_j-\partial_j\omega_i=O(1/|x|^{q+1}), \nonumber \eal
as $|x|\to +\infty $,  where $|\cdot |$ denotes the natural
norm in each coordinate chart, i.e., $|x|^2=\sum_{i=1}^n (x^i)^2$.
Observe that the above growth assumptions on $h$, plus bounds on
  its  second derivatives  (including the  scalar curvature),   are commonly used to define {\em asymptotic
flatness} in   a purely Riemannian setting  (see e.g.
\cite{s} or \cite{bl}). We will not require bounds   neither for the second
derivatives of $h$ nor for its scalar curvature,   as convexity involves only
pregeodesics (i.e. the connection rather than the curvature,   see also Remark~\ref{NewRemark1}).
Consistently, we require only bounds for the  differential of
$\omega$ and not for $\omega$ itself,  because   if an exact form is added
to $\omega$ the pregeodesics remains   unchanged.  For these reasons,  when \eqref{asymptflat}   are
fulfilled on each end we say that the  Randers space is  {\em
geodesically asymptotically flat} (see also Remark
\ref{referee2rem_moved} below).

  Let us focus now  on large spheres
$S^{n-1}(r_0)$, defined in the asymptotic coordinates of each  end $E^{(k)}$, as $|x|^2=r_0^2$,  $r_0  >0$    large enough.
Let
$\phi_\roo^{(k)}(x)=\roo^2-|x|^2$ and
\be \label{edk}
D_\roo^{(k)}=\{x\in E^{(k)}\ |\ \phi^{(k)}_\roo
(x)>0\}.
\ee
  Let $D_{\roo}$  be the domain of $S$ equal to
$\bigcup_{k} D^{(k)}_\roo \cup K$ and $\phi(\equiv \phi_\roo)$   be an
extension of all the $\phi^{(k)}_\roo$'s to $S$. Let us see that
the boundary of $D_{\roo}$ is convex if $\roo$ is big enough. Let
$\gamma=\gamma(s)$ be a geodesic of the Randers metric defined by
$h$ and $\omega$. If $\gamma$ is parametrized with constant
Randers speed $\sqrt
{h(\dot\gamma,\dot\gamma)}+\omega(\dot\gamma)=\mathrm{const.}$
then, in local coordinates, $\gamma$ satisfies the following
equation (compare with \eqref{georandersh}) \beq\label{georand}
\ddot\gamma^l=-\Gamma^l_{ij}
\dot\gam^i\dot\gam^j+\sqrt{h(\dot\gam,\dot\gam)}h^{lm}\big(\partial_m\omega_j
-\partial_j\omega_m\big)\dot\gam^j+ \frac{1}{2}\frac{\de }{\de s}
\left(\log(h(\dot\gam,\dot\gam))\right)\dot\gam^l \, ,  \eeq where
$\Gamma^l_{ij}$ are the components of Levi-Civita connection of
$h$ and $h^{kl}$ is the inverse of $h_{kl}$. Let $x\in\partial
D^{(k)}_{\roo}$ and $y\in T_x\partial D^{(k)}_{\roo}$, arguing as
in the proof of Proposition~\ref{peq} and using \eqref{georand} we
get,
 \be\label{ehessasymp} (H_{\phi})_y(y,y) =-2|y|^2+2x^k h_{kl}
\left(\Gamma^l_{ij}y^iy^j-\sqrt{h(y,y)}h^{lm}\big(\partial_m\omega_j-\partial_j\omega_m\big)y^j\right).
\ee  As $h_{ij}=\delta_{ij}+O(1/|x|^p)$, its inverse $h^{ij}$ is
of the type
 \beq\label{inverse}
h^{ij}=\delta^{ij}+O(1/|x|^p) \eeq and then, recalling the second
condition in \eqref{asymptflat}, $\Gamma^l_{ij}= O(1/|x|^{p+1})$; thus
from \eqref{ehessasymp} we get \be \label{ehessasympInequality}
(H_{\phi})_y(y,y)\leq-2|y|^2+2C \left(\frac{1}{|x|^{p+1}}+
\,\frac{1}{|x|^{q+1}}\right)\,  |x|\, |y|^2 \ee
which
is negative, if $|x|=\roo$ is large enough.  Summing up,  we get
the following

\bpr\label{largesphere}
 In any  geodesically
asymptotically flat Randers manifold (in the sense specified in
formula \eqref{asymptflat})   $\partial D_{r_0}$  is  strongly
convex,  for any sufficiently large enough $r_0$.
 \epr

\section{Stationary spacetimes and causally convex boundaries}\label{s3}
\subsection{Background and SRC} Randers spaces  are deeply related to
the causal structure of stationary Lorentzian manifolds. We start
recalling some basic definitions and notations   (see \cite{beeh,
HE, ms, onei} for further information).

A {\em Lorentzian manifold} is a pair $(L, g_L)$ where $L$ is a
smooth (connected) manifold and $g_L$ a metric on $L$ of index
one, with signature $( + \, , \,\cdots ,+, -)$. A
non-zero tangent vector $v \in
T_z L$, $z \in L$, is said {\em timelike} (respectively {\em
lightlike}; {\em spacelike}) when  $ g_L(v ,v)<0$,
  (respectively  $ g_L(v ,v)= 0$; $ g_L(v ,v)> 0$)
  and {\em causal} if it is timelike or lightlike. A {\em
spacetime } is a
 Lorentzian manifold $(L, g_L)$ endowed with a
time-orientation. The latter is  determined by some timelike
vector field $Y$, so that a causal vector $v \in T_zL$ is  said
{\em future--pointing} (resp. {\em past--pointing}) if $g_L (v, Y)
<0$ (resp. $g_L (v, Y)
>0$). A piecewise smooth curve $z : [a,b] \rightarrow L $ is said
timelike, lightlike or spacelike if so is $\dot z (s)$ at any $s
\in [a,b]$ where it exists. In particular, non-constant geodesics
$z$ are classified according to the sign of $g_L(\dot z,\dot z)$.

A spacetime $(L, g_L)$ is said stationary if
it admits a timelike Killing vector field $Y$. In this case, one
such a $Y$ that points to the future will be chosen and called the
{\em stationary vector field}. When $Y$ is complete and $L$
satisfies a mild causality condition (to be distinguishing, which
lies between causality and strong causality), then $L$ will be
{\em standard stationary} (see \cite[Proposition 3.1]{js}). More
precisely, $L$ splits (in a non-unique way) as a product $L = S
\times\R$, and the metric $g_L$ is given as
\be \label{stst}
\mbox{$g_L$}_{(x,t)}\big(( y , \tau), ( y , \tau)\big) =
\mbox{$g_0$}_x (y,y) + 2 \mbox{$\omega_0$}_x (y) \tau -
\beta(x) \tau^2 \ee for any $(x,t)\in L$, $( y , \tau) \in T_x
S\times \R $, where $g_0$ is a Riemannian metric on $S$, $\omega_0$
and $\beta$ are, respectively, a smooth vector field and a smooth
positive function on $S$, and, moreover $ Y=
\partial_t $.

In what
follows, $Y$ will be a prescribed complete stationary vector field
in a distinguishing spacetime so that the splitting \eqref{stst}
holds, and the effect of changing the slice $S$ in this splitting
will be taken explicitly into account. So a piecewise  smooth
causal curve $z(s) = (x(s),t(s))$ is future--pointing (resp.
past--pointing) if and only if  $\beta \dot t - \omega_0 (\dot x )
>0$ (resp. $\beta \dot t - \omega_0 (\dot x )
< 0$).

Projections on $S$ of lightlike geodesics of a  standard
stationary spacetime are pregeodesics  for a Randers metric.
Indeed, a lightlike curve $z (s ) = ( x (s) , t (s))$,
parametrized on a given interval, say $s\in I=[a,b]$,  satisfies
\begin{equation}\label{eqlum}
g_0 \big(\dot x (s), \dot x (s)\big) + 2 \omega_0 \big(\dot x (s)\big) \dot t (s)  - \beta (x (s)) \dot t^2 (s)=0
\end{equation}
Taking into account the zeros in $\dot t$  of this equation, define the  Finsler
 metric
\begin{equation}\label{efermat}
F(y)=  \left(\big( \omega_0 (y )^2 +  \beta  g_0 (y ,
y)\big)^{1/2} +\omega_0 (y )\right)\frac{1}{\beta},
\end{equation}
for all $ y \in TS$, as well as its reverse metric
$\tilde F$. According to \cite{cym}, these Finsler
metrics are called {\em Fermat metrics}. Notice that $F$ is of
Randers type, $F=\sqrt{h}+\omega$ with:
\begin{align}
h(y,y)& =  \frac{1}{\beta^2} \omega_0 (y )^2
+ \frac{1}{\beta}  g_0 (y , y)       \label{f1}     \\
\omega (y) & =  \frac{1}{\beta} \omega_0 (y) \label{f2}
\end{align}
for $(x,y) \in TS$. Now, from \eqref{eqlum}, we have two
possibilities for $\dot t$:
 \beq\label{dottfuture} \dot t  = F(\dot x) \quad \quad \dot t =
 -\tilde F(\dot x),
 \eeq
the first equality if $z$ is future--pointing, and the second one
if it is past--pointing. Then, putting, $z(a)=(p,t_p)$ the {\em
arrival time} of the curve $z$, that is,  the value of the $t$
coordinate at $z(b)$, is given by:
\be \label{ar} T (z) = t_p +
\inte F(\dot x) \de s \quad \quad \tilde T (z) = t_p - \inte \tilde
F(\dot x) \de s \ee depending, resp., on if $z$ is future or past--pointing.

The Fermat principle states that $z$ is a critical point of the
(future or past) arrival time if and only if $z$ is a (future or
past) lightlike pregeodesic (i.e. a geodesic up to a
reparametrization) for the spacetime, see \cite{cym}. However, it
is obvious from \eqref{ar}, that these critical curves coincide
with the pregeodesics for $F$ and $\tilde F$. Choosing an
appropriate parametrization we have finally:

 \begin{prop} \label{fp}
Let $(L,g_L)$ be a standard stationary spacetime. A curve of the
type $z (t) = (  x (t),  t ) \in L$, $t \in [t_0,t_1]$, is a
future--pointing, lightlike pregeodesic
if and only  if $x(t)$,  $t\in [t_0,t_1]$, is a unit geodesic for
the Fermat metric $F$ defined by \eqref{efermat}.
\end{prop}
An analogous statement holds for past--pointing lightlike
geodesics and geodesics of the Randers metric $\tilde F$.

However, the relation between stationary spacetimes and Randers
metrics is much deeper and, in particular, involves the full
causal structure of the spacetime \cite{cys}.
Recall that given two points ({\em events}) $w,z \in L$, $w$ is
{\em causally related} to $z$ ($w \leq z$) if either $w=z$ or
there exists a future--pointing, causal curve from $w$ to $z$. The
{\em causal future} of $w \in L$ is the set $J^+ (w) = \left\{ z
\in L \mid w \leq z \right\}$. An analogous definition holds
substituting future--pointing curves with past--pointing ones,  so
obtaining the {\em causal past} of $w$, $J^-(w)$. Spacetimes can
be classified according to their increasingly better causal
properties getting the so called {\em causal ladder} of spacetimes
(see \cite{beeh,ms}). In particular, a spacetime is {\em causal}
when it does not contain any closed causal curve, {\em causally
simple} when it is causal and, for any $w\in  L$, the   causal
futures and pasts $J^\pm (w)$ are closed, and {\em globally
hyperbolic} when it is causal and $J^+(w)\cap J^-(z)$ is compact
for all $w, z$ (for these definitions, recall \cite{BS}). Among
other properties, one has:

\begin{teo}\label{tscr} \cite{cys} Let $L = (S \times \R, g_L)$  be a  standard
stationary spacetime and $(S , F)$ be its associate Randers space
as in \eqref{efermat}. Then:

(1) $L$ is causally simple if and only if the space $(S , F)$ is
convex.

(2) $L$ is globally hyperbolic if and only if the closed
symmetrized balls of the space $(S , F)$ are compact.
\end{teo}
That is, the weakening of the global hyperbolicity condition into
causal simplicity for a stationary spacetime is parallel to the
weakening of the compactness of the closed symmetrized balls into
convexity for $(S , F)$.

\bere \label{NewRemark1}
As suggested above, the standard stationary splitting is not
uniquely determined by, say, the timelike Killing vector field   $Y$.  In fact, it can be changed by
replacing the spacelike hypersurface $S$ by a new one $S'$. Such a $S'$ can be written as a
(spacelike) graph $S'=\{(x,f(x) ) \}$ for some function $f$ (whose differential has  $F$-norm smaller than $1$).
The Fermat metric $F'$ associated to $S'$ satisfies  $F'=F-{\rm d} f$  (with natural identifications,
see details in \cite[Prop. 5.9]{cys}). That is,  the metric $h$ remains invariant (in fact, $h$ is
the metric induced  from the orthogonal distribution to $\partial_t$,  up to the conformal
factor $1/\beta$,   and the one form $\omega$ is ``gauge transformed'' as
$\omega'=\omega-{\rm d} f$.

As
emphasized in \cite{cys},  the hypotheses in Theorem \ref{tscr}  are invariant under such a
change,  as they are related to the conformal geometry of the spacetime  and, then,  they are independent of the
choice of the standard splitting.  Because of this same reason,  the conditions to be studied here
are typically independent of the change $\omega \mapsto \omega -{\rm d} f$. In fact, this is obvious
in Proposition~\ref{peq} and in formula \eqref{asymptflat}, as only conditions on ${\rm d} \omega$ (and not on $\omega$
itself) are involved.
\ere

\subsection{Standard stationary    domains and light-convexity}
 Now, choose  a domain $D$ of $S$ with smooth boundary
$\partial D$ and consider the {\em stationary domain} $D \times
\R$ as a domain of $M=S \times \R$ with boundary $\partial D
\times \R$.  The Finslerian notions of convexity for the
boundaries of domains can be extended to the Lorentzian case, and
we can speak on the {\em infinitesimal} or {\em local convexity}
of $ \partial D\times\R$ (recall Remark \ref{manyconvex} and \ref{local
convexity}). However, in the Lorentzian context  it  is natural to
take into account the causal tripartition of the tangent vectors.
Even more,  also the structure of standard  domains for stationary
spacetimes will be taken into account here. So, consider a
function $\Phi : S \times \R \rightarrow \R$, $\Phi (x,t) = \phi
(x)$ such that \beq\label{Phi}
\begin{cases}
          \Phi^{-1}(0) =  \partial D\times \R \\
          \Phi > 0&\mbox{on $ D \times \R$ }\\
          {\rm d} \Phi(z)\not= 0 &\mbox{for every  $z \in \partial D \times \R$. }
 \end{cases}
\eeq We say that $\partial D \times \R$ is  {\em infinitesimally
time--convex}   (respectively {\em light--convex}) if for any
$z=(x,t) \in
\partial D \times \R$ and for any  timelike (respectively
lightlike) vector $(  y , \tau)  \in   T_x
 \partial D\times  \R$, one has  $H^{g_L}_\Phi \big(  (  y , \tau) , (  y , \tau) \big)
\leq 0$, where $H^{g_L}_\Phi$ denotes the Hessian of $\Phi$ with
respect to the Lorentzian metric $g_L$   (recall Remark~\ref{manyconvex}). Whenever the last
inequality is satisfied with strict inequality, we say that
$\partial D \times \R$  is
{\em strongly time--convex} (respectively {\em strongly
light--convex}).   The infinitesimal
light-convexity can be characterized directly in
terms of the corresponding Fermat metric as follows.
\begin{teo} \label{lighteq}
Let $(L,g_L)$ be a standard stationary spacetime and let $D$ be a
domain of class $C^{2}$ of $S$. Then $(\partial D; F)$ is
infinitesimally
convex (resp.  strongly convex)  if and only
if $(\partial D \times \R; g_L )$ is infinitesimally light--convex
(resp.  strongly light-convex).
\end{teo}
\begin{proof}
Observe that the metric $g_L$ can be written as
\[
 g_L \big(( y ,\tau), ( y ,\tau)\big) = \Big( h (y,y) - \big( \tau - \omega ( y )
\big)^2 \Big)\beta,
\]
where $h$ and $\omega$ are defined in \eqref{f1} and \eqref{f2}. Using this expression, we can
easily compute the geodesic equations of $(L,g_L)$. Denoted by $z=z(s)=(x(s),t(s))$ a geodesic of
$(L,g_L)$, its components $x$ and $t$ satisfy the equations
\beq\begin{cases}
\big(\dot t-\omega (\dot x)\big)\beta=\mathrm{const.}:=C_z\\
\frac 1 2\nabla^h\beta\Big(h(\dot x,\dot x)-\big(\dot t-\omega (\dot x) \big)^2\Big)=
\nabla^h_{\dot x}\big(\beta \dot x\big)-C_z \widehat{ \de \omega } (\dot x)
\end{cases}\label{geoeq}
\eeq where $\widehat{ \de \omega }$ is  the $(1,1)$--tensor field
$h$-metrically associated to $ \de \omega$  as in Proposition
\ref{peq} above.

If $z$ is a lightlike geodesic, then  $h (\dot x,\dot x) - \big( \dot t - \omega ( \dot x ) \big)^2 =0$ and
the second equation in \eqref{geoeq} becomes
\[\nabla^h_{\dot x}\dot x=-\frac{h(\nabla^h\beta,\dot x)}{\beta}\dot
x+\frac{C_z}{\beta}\widehat{ \de \omega } (\dot x).\]
If $z$ is future--pointing,  from the first equation in \eqref{dottfuture}  and in \eqref{geoeq}     we have
\[\frac{C_z}{\beta}=\dot t-\omega (\dot x)=\sqrt{h(\dot x,\dot x)}.\]
Hence the equation satisfied by the $x$ component of a
future--pointing lightlike geodesic is \beq\label{eqlightgeo}
\nabla^h_{\dot x}\dot x=-\frac{h(\nabla^h\beta,\dot x)}{\beta}\dot
x+\sqrt{h(\dot x,\dot x)}\widehat{ \de \omega}(\dot x). \eeq Arguing as
above, we can see that the $x$ component of a past--pointing
lightlike geodesic satisfies  equation \eqref{eqlightgeo} with the
$-$ sign instead of $+$ in the right hand side. Let
$(x_0,t_0)\in\partial D\times \R$ and $(y_0,\tau_0)\in
T_{(x_0,t_0)}(\partial D\times \R)$ be a lightlike vector.
Consider the lightlike geodesic $z=z(s)=(x(s),t(s))$ such that
$z(0)=(x_0,t_0)$ and $\dot z(0)=(y_0,\tau_0)$. Since $(\Phi\circ
z)''(s)=H^{g_L}_{\Phi}(\dot z(s),\dot z(s))$ and $\Phi \circ
z=\phi\circ x$,   by using  \eqref{eqlightgeo}, \eqref{ddotrho}
and recalling also that $y_0$ is orthogonal to $\nabla \phi(x_0)$,
we get \[H^{g_L}_\Phi \big( ( y_0 ,\tau_0) ,  ( y_0
,\tau_0) \big)= H^h_\phi (y_0,y_0)\pm\sqrt{h(y_0,y_0)}\, \de \omega
(y_0 ,\nabla^h\phi), \]
with the $+$ sign if $(y_0,\tau_0)$ is
future--pointing and the $-$ sign otherwise. Hence the thesis
follows from Proposition~\ref{peq}.
\end{proof}
\begin{remark}\label{reverse}
There are several subtleties to be taken into account:

(1)  Consistently with Remark \ref{NewRemark1}, the hypotheses of the theorem are invariant
under the change of  the  standard stationary splitting;  indeed, if $D$ is changed to the domain $D'=\{(x,f(x)): x\in D\}$ for some function $f\colon D\to \R$, then $D'\times R$ is clearly equal to $D\times R$ and, by Remark~\ref{NewRemark1}, $\partial D'$ will be infinitesimally convex w.r.t. $F'$. Moreover, the  equivalence between the
convexity for $F$ and $\tilde F$
of $\partial D$ in Proposition \ref{peq} (and also of the domain
$D$), is consistent with the notion  of light-convexity, which
makes no difference between future and past--pointing lightlike
vectors.

(2)  Consistently with  Theorem \ref{lighteq} and Example
\ref{r2}, the light-convexity of $(\partial D \times \R; g_L )$ is
not related to the convexity of $(\partial D; g_0 )$. Indeed the
Randers metric in Example \ref{r2} can be regarded as the Fermat
metric associated to $(\R^2 \times \R , g_L)$ where
$\mbox{$g_L$}_{(x^1 , x^2 , t )}\big( (y^1 , y^2 , \tau), (y^1 ,
y^2 , \tau)\big)  = (1- f(x^1 )^2)(y^1)^2 + (y^2)^2 + 2 f(x^1 )
y^1 \tau - \tau^2$.

(3)
As in the Finsler case,   the notion
of infinitesimal convexity for  a hypersurface of the type $\partial D\times \R$ in
a stationary spacetime
is a  natural extension   of the analogous
convexity for a hypersurface  in a Riemannian manifold.   Moreover, the latter notion  is
trivially extensible to any
 embedded   hypersurface $H$ in any Lorentzian manifold\footnote{Usually, one has to assume that the hypersurface is also non-degenerate but, since  our definition (recall Remark~\ref{manyconvex}) does not involve the second fundamental
form of $H$, the non-degeneracy assumption can be dropped  (cf. \cite[Remark 3]{caponiogelogra}). Clearly, in the non-degenerate case, one recovers the usual condition  about the sign of the second fundamental form.},
not only the stationary ones.
Notice  that, as the
Hessian depends on the Levi-Civita connection rather than on the
metric, the  inequality remains  in the same direction  as  in the
positive-definite case (that is,  no change of sign is required
for timelike directions). However, as pointed out in the
stationary case, this notion of convexity can be weakened
according to the causal character of the involved vectors, that
is, we say that $H$ (expressed locally as $\phi^{-1}(0)$ for some
$\phi$ as in \eqref{defb1}) is {\em infinitesimally light- ({\rm
resp.} time-, space-) convex}  if $(H^{g_L}_\phi)_z(v,v)\leq 0$
for any $(z,v)\in TH$, with $v$ lightlike (resp. timelike,
spacelike).

(4) The notion of {\em local convexity}, explained in Remark
\ref{local convexity}, is also trivially extensible to the
Lorentzian case from the  Riemannian or Finslerian ones, and its
equivalence with infinitesimal convexity  can be also proved by transplanting the
technique in \cite{bcgs}, see
\cite{caponiogelogra}.
Again, in the general Lorentzian case, we can define also {\em
local time-, space- or light-convexity} by considering only
geodesics of the corresponding type. However, its equivalence with
the corresponding infinitesimal notions is subtler, see
\cite{caponiogelogra}.
\ere
Regarding the last point above, recall that lightlike vectors are
points in the boundary of the (open) subsets of both, time and
spacelike vectors and, indeed, if a hypersurface is
infinitesimally time- or space-convex, then it is also
infinitesimally light-convex by continuity. But, in principle, we
cannot state that  it is also locally convex with respect to
lightlike geodesics. Moreover,  in principle,  the proof of the
equivalence between local and infinitesimal convexity in
\cite{bcgs} cannot be extended to local and infinitesimal
lightlike convexity (see \cite[Remark 6]{caponiogelogra}).   Nevertheless, in the case of a {\em standard
timelike hypersurface} $H$ in a standard stationary spacetime (i.e
$H= H_S\times R$, where $H_S$ is a hypersurface in $S$),
Proposition \ref{fp} and Theorem \ref{lighteq} give the
equivalence also in the lightlike case.

\bco\label{lightequiv} Let $(S\times\R, g_L)$ be a standard
stationary spacetime, $H_S$ be a  $C^{2}$  embedded 
hypersurface in $S$.
The hypersurface
$H=H_S\times\R$ in $S\times\R$
 is  infinitesimally light-convex
if and only if it is  locally light-convex. \eco
\begin{proof}
The implication to the left follows easily as in the Riemannian
setting. To prove that  infinitesimal light-convexity implies
local light-convexity, for any $x_0\in H_S$ take a neighborhood
$U_S$ and a function $\Phi\colon U_S\times\R\to\R$,
$\Phi(x,t)=\phi(x)$, $\phi\colon U_S\to\R$ (which  satisfies
\eqref{Phi} with $H_S$ in place of $\partial D$ and
$\phi^{-1}\big((0,+\infty)\big)\cap U_S$ replacing $D$), such that
$H^{g_L}_{\Phi}\big((y,\tau), (y,\tau)\big)\leq 0$, for all
$(y,\tau) \in T (H_S\times\R)$  with $(y,\tau)$ lightlike. By
Theorem~\ref{lighteq}, $H_S$ is infinitesimally convex in $U_S\cap
H_S$ for  the Fermat metric   and, then, by \cite[Theorem
1.1]{bcgs} (recall also Remark~\ref{c2}) it is locally convex in the same neighborhood of $x_0$
in $H_S$ with respect to the geodesics of both, the Fermat metric
in \eqref{efermat} and its reverse metric $\tilde F$ (recall
Remark~\ref{reverse}(1)). This means that for each $x\in U_S\cap
H_S$ the exponential maps with respect to $F$ and $\tilde F$ send
the vectors in a neighborhood of the origin in $T_{x_0} H_S$ into
$\phi^{-1}\big((-\infty, 0]\big)\cap U_S$. From
Proposition~\ref{fp}, the exponential map of $g_L$ maps future and
past--pointing  lightlike vectors in a neighborhood of the origin
in $T_{(x,t)}(S\times\R)$ into $\big(\phi^{-1}\big((-\infty,
0]\big)\cap U_S\big)\times\R$ so that $H$ is locally light-convex
at any
point of $(U_S\times\R)\cap H$.
\end{proof}
\subsection{Time-convexity}\label{timeconvexity}
Randers metrics  can be  also used  to characterize the  convexity of the
boundary of a (stationary) region of a standard stationary
spacetime with respect to timelike geodesics. Actually,
time-convexity  can be reduced to light-convexity in a    suitable
one-dimensional higher product manifold   (see \cite[Subsection
4.3]{cym}), such a trick is valid in a much more general setting
for any Lorentzian   metric \cite{CapMin04, minsan}.   We start by
pointing out some technical properties.

Consider a standard stationary spacetime $(L=S\times\R,g_L)$ as in
\eqref{stst}. Let $\R_u\times S$ denote the product manifold
$\R\times S$ where the subscript $u$ means that the   natural
metric $+du^2$ is considered on $\R (\equiv \R_u)$.
Put $ L_1 =
  (\R_u\times S)\times  \R  \equiv  \R_u\times L$
and denote the usual projections:
\[\begin{array}{ll}
 \Pi_S: \R_u\times S \rightarrow S,
& \Pi_u:\R_u \times S\rightarrow \R_u, \\
\Pi_1\colon\R_u\times L\to \R_u, &\Pi\colon \R_u\times L\to L.
\end{array}\]
Now, endow  the manifold $ L_1 =
  \R_u\times L$ with the standard stationary Lorentzian metric $g_{L_1}$ defined as
\beq\label{enne} g_{L_1}=\Pi_1^*du^2+\Pi^*g_L. \eeq Obviously, a
curve $s\mapsto (u(s), x (s) ,t (s) )$ is a geodesic in $( L_1 ,
g_{L_1} )$  iff 
 $s\mapsto z (s) :=  (x (s) , t(s))$ is a geodesic for $g_L$ and $\ddot u (s) =0$. Then, a
lightlike geodesic for $g_{L_1}$ parameterized with a constant
$\dot u (s)=: \ell$  satisfies $g_L ( \dot z , \dot z ) =  -\ell^2$, so
that $z=z(s)$ is an affinely parametrized timelike geodesic of
$(L,g_L)$,  provided that $\ell\neq 0$.   The Fermat metric for $(L_1=(\R_u\times
S)\times \R, g_{L_1})$ takes the form\footnote{Notice that, for
lightlike geodesics, the construction of the Fermat metric was
conformally invariant and, so, the elements $h, \omega$ where
normalized so that $\beta$ could be regarded as an overall
conformal factor, eventually equal to 1. However, this conformal
invariance does not hold for timelike geodesics, and it is
emphasized by means of the subscript $\beta$.}: \beq\label{ran1}
F_{\beta}=\sqrt{\Pi_S^*h + \frac{\Pi^*_u du^2}{\beta  \circ \Pi_S
}} + \Pi_S^*\omega = \sqrt{h_\beta} +\omega_1. \eeq where $h,
\omega$ are as in \eqref{f1}, \eqref{f2}, and:
$$ h_\beta=\Pi_S^*h
+ \frac{\Pi^*_u du^2}{\beta \circ \Pi_S }, \quad
\omega_1=\Pi^*_S\omega, \quad \mbox{on} \; \R_u\times S.$$  The
arrival time of a future--pointing timelike geodesic $z(s)=
(x(s),t(s))$, parameterized on $[a,b]$, connecting a point $(p,
t_p )$
 of $ S \times
\R$  to a  line  $l (\tau) = (q,  \tau) \in S \times \R$ and such
that $g_L (\dot z , \dot z ) = - \ell^2$  is given by \be \label{ar1}
T(z) = t_p + \int_a^b \left( \sqrt{h( \dot x , \dot x )+ \frac
{\ell^2}{\beta\circ x } }+ \omega  (\dot x) \right) \de s. \ee

For a given domain $D$ of class $C^2$ of $S$ we can study the infinitesimal
convexity of $\partial (\R_u\times D)= \R_u \times
\partial D$ with respect to $F_\beta$  in \eqref{ran1}. Take
$\phi$ as in (\ref{defb1}) globally defined on $S$ (Remark
\ref{local convexity}), and set $\phi_1:\R_u\times S\rightarrow
\R$, as $\phi_1 (u, x ) = \phi (x )$. By Proposition~\ref{peq},
$(\R_u \times
\partial D;F_\beta)$ is convex if and only if
\beq\label{timeconvex} H^{h_\beta }_{\phi_1}\big((v,y),(v,y)\big)
+\sqrt{h_\beta \big((v,y),(v,y)\big)} \de \omega_1  \big((v,y) ,
\nabla^{h_\beta }\phi_1\big)\leq 0, \eeq for all $(u,x)\in
\R_u\times\partial D$ and $(v,y)\in \R_u\times T_x\partial D$. Trivially, $\de \omega_1 \big( (v,y) ,
\nabla^{h_\beta }\phi_1\big) = \de \omega \big(  y  , \nabla^{h
}\phi\big)$.  Moreover, as $(\R_u \times D, h_\beta) $ is a warped product,  taking into account  geodesic equations  in  this kind of manifolds (see e.g.
 \cite[Ch.7, Proposition 38]{onei}) and by using \eqref{evH},  it is not difficult to evaluate the Hessian of $\phi_1$ with respect to $h_\beta$ obtaining
\[H^{h_\beta }_{\phi_1}\big((v,y),(v,y)\big)=H_\phi^h (y,y)  - \frac{h (
\nabla^h\phi , \nabla^h  \beta )}{2 \beta^2 }v^2.\] Summing up,
substituting these expressions in  \eqref{timeconvex}, one has:

\begin{lem} $(\R_u \times
\partial D;F_\beta)$ is infinitesimally
convex if and only if
\beq H_\phi^h(y,y) - \frac{v^2}{2 \beta^2} h (\nabla^h\phi,
\nabla^h \beta) +\sqrt{h(y,y) + \frac{v^2 }{\beta}} \;  \de \omega ( y
, \nabla^{h }\phi) \leq 0 \label{tcon} \eeq for any $ y \in T\partial D$,  $v \in \R$.
\end{lem}
Likewise the  case of lightlike geodesics,  the following result
holds.
\begin{teo} \label{peq2}
Let $(S\times\R, g_L)$ be a standard stationary spacetime    and
let  $D$ be a domain of class  $C^2$ of $S$. Then $(\partial D
\times \R; g_L )$ is infinitesimally time-convex   (resp. strongly time-convex)  if and only if
$(\R_u\times\partial D; F_\beta  )$ is infinitesimally
convex  (resp. strongly convex).
\end{teo}
\begin{proof} Let us check that (\ref{tcon})
holds if and only if  the Lorentzian Hessian of $\Phi$  is
non-positive on timelike vectors on the tangent bundle of
$\partial D\times \R$. To this end we argue as in the proof of
Theorem~\ref{lighteq}, with $(L_1,g_{L_1})$ replacing $(L,g_L)$.
This time since  $z=z(s)= (x (s) , t (s))$  is timelike, the
second equation in \eqref{geoeq} becomes \beq
-\frac{\nabla^h\beta}{2\beta} v^2=\nabla^h_{\dot x} \big(\beta
\dot x\big)-C_z \widehat{ \de \omega }(\dot x)=
\beta\nabla^h_{\dot x}\dot x+h(\nabla^h\beta,\dot x)\dot x-C_z
\widehat{ \de \omega } (\dot x),\label{timegeoeq}\eeq where
$-v^2=g_L(\dot z,\dot z)\neq 0$. As $z$ is
future--pointing $\dot t=\omega(\dot x)+\sqrt{h(\dot x,\dot
x)+\frac{v^2}{\beta}}$ and then
$C_z/\beta=\sqrt{h(\dot x,\dot x)+\frac{v^2}{\beta}}$.
Recalling that $\Phi\circ z=\phi\circ x$,
 $$H^{g_L}_{\Phi}\big(\dot
z(s),\dot z(s)\big)= (\phi\circ z)''(s) = H_\phi^h \big(\dot x (s)
, \dot x (s)\big) + h \big(\nabla^h \phi (x (s) ), \nabla^h_{\dot
x (s) } \dot x (s)\big). $$  So, computing $\nabla^h_{\dot x}\dot
x$ from \eqref{timegeoeq} the left-hand side of \eqref{tcon} is
equal to $H^{g_L}_{\Phi}\big(\dot z(0),\dot z(0)\big)$.
\end{proof}
\begin{remark}\label{timeequiv}
 The previous result yields  a chain of
equivalences which, in particular, shows the equivalence between
the infinitesimal and local time-convexities for $(\partial D
\times \R,g_L)$. In fact, from the construction of $g_{L_1}$
above, $(\partial D \times \R,g_L)$ is locally time-convex iff
$(\R_u\times
\partial D\times \R,g_{L_1})$ is  locally
light-convex. By  Corollary~\ref{lightequiv},  this holds iff
$(\R_u\times
\partial D\times \R,g_{L_1})$ is infinitesimally light-convex, and  by
Theorem~\ref{lighteq}, iff  $(\R_u \times \partial D;  F_\beta )$
is infinitesimally convex. Finally, by Theorem \ref{peq2} this
holds iff $(\partial D \times \R,g_L)$ is infinitesimally
time-convex, as required.
\end{remark}
In order to apply
Theorem~\ref{peq2}, the following characterization is useful.
We emphasize that it is  independent of the choice of the standard stationary splitting, in agreement with Remarks~\ref{NewRemark1} and \ref{reverse}(1).
\begin{prop}\label{Rbetaconvexity2}
Consider a Randers space $(S,R)$ as in \eqref{ran} and, for any
function $\beta>0$, the Randers space $(\R_u\times S, R_\beta )$
where $R_\beta$ is constructed as $F_\beta$ in \eqref{ran1}. Let
$D$ be a domain of class $C^{2}$ of $S$. Then
$(\R_u\times
\partial D; R_\beta )$ is infinitesimally  convex if and only if the following
 three  conditions hold  for all $ x\in
\partial D$:
\begin{itemize}
 \item[i)]
 $(\partial D; R)$ is infinitesimally convex, i.e. (Prop.
\ref{peq}),
\[H_\phi^h (y,y)
+ \sqrt{h(y,y)}\,\big| \de \omega (y,  \nabla^h \phi)\big|
 \leq 0 ,\]

 \item[ii)]$\nabla^h\beta$ does not point  outside  $D$ at $x$,
 i.e.
\beq\label{necessary2} 0\leq h_x (\nabla^h\phi, \nabla^h  \beta ),
\eeq
 \item[iii)] for each $y\in T\partial D$,  {\em
either} \beq \de \omega (y,\nabla^h\phi)^2  +  
 \frac{h(\nabla^h\phi,\nabla^h\beta)}{\beta}H^h_{\phi}(y,y)  \leq 0 \label{roots} \eeq
{\em or} $h(\nabla^h \phi,\nabla^h\beta)>0$ and \beq
2H^h_{\phi}(y,y)+\frac{\beta}{h(\nabla^h\phi,\nabla^h\beta)}\de
\omega(y,\nabla^h\phi)^2+\frac{h(\nabla^h\phi,\nabla^h\beta)}{\beta}h(y,y)
\leq 0.\label{quadraticform2} \eeq
\end{itemize}
\end{prop}
\begin{proof}
Recall that the convexity of  $(\R_u\times \partial D; R_\beta )$
is equivalent to \eqref{tcon}, and put: \baln
\lambda^2=\frac{v^2}{\beta}, &&r^2=h(y,y),&&a\,
r^2=H^h_\phi(y,y),&&b= \frac{h(\nabla^h
\phi,\nabla^h\beta)}{\beta},&&d\, r=\de \omega(y,\nabla^h\phi) \ealn
All these elements except $d$ remain invariant if $y$ is changed
by $-y$. So, define the functions $f_\pm\colon[0,+\infty)\times[0,+\infty)\to
\R$,
$$f_\pm(r,\lambda)=a r^2 -\frac{b}{2}\lambda^2 \pm d\,
r(r^2+\lambda^2)^{1/2}.$$ Then, \eqref{tcon} holds if and only if
\beq\label{31} f_+(r,\lambda)\leq 0 \quad \hbox{and} \quad
f_-(r,\lambda)\leq 0 \eeq for all $r,\lambda\geq0$.  Evaluating
these inequalities at $\lambda=0$, one has $ r^2(a\pm d)\leq 0 $, which
shows the necessity of i) and  gives also
\be\label{e1}
a\leq 0 \quad \text{and} \quad d^2\leq a^2. \ee
Evaluating the same inequalities at $r=0$, we
have \beq\label{32}
0\leq b, \eeq
which proves the necessity of
the  condition  ii).  So, assuming that \eqref{32} holds,
the conditions \eqref{31} are equivalent to
\[
d^2 r^2(r^2+\lambda^2) \leq \left( ar^2-\frac{b}{2}\lambda^2\right)^2,
\] that is:
\beq\label{34} 0\leq (a^2-d^2) r^4
-(d^2+ab)r^2\lambda^2 +\frac{b^2}{4}\lambda^4. \eeq
Finally, under the previous  necessary conditions \eqref{e1}-\eqref{32}, equation  \eqref{34}
holds iff either
its roots,  as a polynomial in the variable $r^2$\, are
non-positive (i.e. $d^2\leq -ab$, in agreement with
\eqref{roots}), or if its discriminant,
 is non-positive  (this can happen only
if $h(\nabla^h \phi,\nabla^h\beta)\neq 0$ and, under
\eqref{necessary2},  it is equivalent to
\eqref{quadraticform2}), as required.
\end{proof}
As an application of the previous result to be applied later,
recall:

\bco\label{betabehaviour} Consider   a geodesically asymptotically   flat
end $E^{(k)}$ of a Randers space  (see \eqref{asymptflat}) and a ball
$D_{\roo}^{(k)}$ as in \eqref{edk}. The   boundary of the  domain $(\R_u\times
D_{\roo}^{(k)}; R_{\beta})$ is   strongly  convex for large $\roo$ if $\beta$
satisfies, as $|x|\to\infty$: \bal\label{betasympt}
&\beta=C_1+O(1/|x|^{q'}),\quad\quad\partial_i\beta=O(1/|x|^{q'+1}),\quad\quad\partial_r\beta
\sim-\frac{C_2}{|x|^{q'+1}}, \eal for some $C_1,C_2>0$, and $q'\in
[0,2q)$, where $\partial_r $ is the vector field
$\partial_r=\frac{x^i}{|x|}\partial_i$ \eco

\begin{proof}
 As the Randers manifold is asymptotically flat, from
Prop.~\ref{largesphere} $(\partial D_{\roo}^{(k)};R)$ is
infinitesimally convex, i.e. $\mathrm i)$ of
Proposition~\ref{Rbetaconvexity2} is  satisfied.
Recalling  \eqref{inverse}, we get
\beq\label{prodotto}
h(\nabla^h\phi_{\roo}^{(k)},\nabla^h\beta)=h^{ij}\partial_i\beta\partial_j\phi_{\roo}^{(k)}=
-2|x|\partial_r\beta+O(1/|x|^{p+q'}),\eeq
 hence
\eqref{necessary2} is  satisfied for $|x|$ large enough. From
\eqref{asymptflat}, \eqref{betasympt} and \eqref{prodotto}, we get
\baln \lefteqn{\beta\de \omega (y,\nabla^h\phi)^2 +H^h_{\phi}(y,y)
h(\nabla^h\phi,\nabla^h\beta)\leq}&\\
&\leq\frac{C_3|x|^2|y|^2}{|x|^{2q+2}}+4|y|^2|x|\partial_r\beta+\frac{C_4|x||y|^2}{|x|^{p+1}}\Big|2|x|\partial_r\beta+O(1/|x|^{p+q'})\Big|\\
&\leq|y|^2\left(\frac{C_3}{|x|^{2q}}-\frac{2C_2}{|x|^{q'}}+\frac{C_5}{|x|^{p+q'}}+\frac{C_6}{|x|^{2p+q'}}\right)<0,
\ealn for all $y\in TS_\rho$,  that is
\eqref{roots} is also satisfied, for all $y\in T_x \partial
D^{(k)}_{\roo}$, provided that $|x|=\roo$ is large enough.
\end{proof}

\section{Applications to asymptotically flat stationary spacetimes}\label{appflat}
\subsection{The notion of asymptotically flat stationary spacetime}\label{KN}
As an application of the  results  in Subsection~\ref{asflat} and
of Proposition~\ref{Rbetaconvexity2} we will consider in the next
subsection   spheres of large radius in the spacelike slice $S$ of an asymptotically flat
stationary spacetime  and, in particular, of the  stationary region
of the Kerr spacetime. But, previously, the notion of
asymptotically flat spacetime is revisited now in the framework of
stationary spacetimes.

 Roughly speaking, for asymptotically flat spacetimes the curvature becomes negligible at large
distances from some region, so that  the geometry becomes
Minkowskian there. This is commonly expressed by assuming the
existence of suitable asymptotic coordinates, so that the
difference between the original metric and Minkowski one (plus
their first and second derivatives) falls-off at an enough fast
radial rate. Penrose conformal boundary \cite{Pen} allows to
circumvent the problem of suitably defining and evaluating limits
with a truly coordinate-free definition of asymptotic flatness, as
done explicitly by Geroch \cite{Ge}. This intrinsic procedure
succeeded (see for example \cite{HE, wald, Frau}) but,  at any
case, the appropriate fall-off behavior in coordinates must be
recovered at some step.

So, in the particular case of  stationary spacetimes, the usual
definition of asymptotic flatness  implies the existence
of a standard stationary splitting with respect to
some spacelike hypersurface $S$ such that for some compact set
$K\subset S$, $S \setminus K$ is a disjoint union of ends,
$E^{(k)},k=1,\dots , m$, each one admitting asymptotic coordinates
$x=(x^1,\ldots,x^n)$ where:
\bal\label{asymptflat3}
&\big(h_{ij}-\delta_{ij}\big)+|x|\partial_{l} h_{ij} +|x|^2\partial^2_{kl} h_{ij}=O(1/|x|^{\alpha}),\nonumber\\
&\omega_j+|x|\partial_{l} \omega_j +|x|^2\partial^2_{kl} \omega_j=   O(1/|x|^{\alpha}),\\
&(\beta-1)+|x|\partial_{l}\beta+|x|^2\partial^2_{kl}\beta=O(1/|x|^{\alpha}), \nonumber
\eal
with $\alpha>1/2$ (see \cite[p.13--14]{BeiSch00} and also \cite{BeiChr96}).

In the case of standard static spacetimes, the integrability of
the orthogonal distribution $Y^\perp$ to the static vector $Y$
selects a (positive definite) Riemannian manifold  --so that the
spacetime notion of asymptotic flatness is simplified into the
more elementary notion of asymptotic flatness for Riemannian
manifolds. Therefore, in this setting, the definition becomes
satisfactory, and it is used systematically for  problems relative
to positive mass and the Riemann-Penrose conjecture (see
\cite{BrCh, bl} and references therein).


Hereafter, as only properties of the geodesics will be required,
we will not need to impose any bound for derivatives of order
greater than $1$. Moreover, according to \eqref{asymptflat}, the
rate of fall-off at infinity can be arbitrarily slow.  As in
Subsection~\ref{asflat},  we will add the surname ``geodesically''
in the definition in order to distinguish  our scarcely
restrictive bounds from the more usual  ones\footnote{Recall that,
for example, in order to have a well defined, unique and
non-necessarily vanishing ADM mass of  a $3$ dimensional
Riemannian manifold $(S,h)$,  the decay rate of $h$ -- which
involves also the Ricci tensor -- must be of order not less than
$1/2$ and not greater than $1$, \cite[Theorems 4.2 and
4.3]{bart}).} \eqref{asymptflat3}.

\begin{defn}\label{referee2def} Let $L=S \times \R$ be a standard stationary
spacetime with prescribed Killing vector field $  Y  =\partial_t$ as
in (\ref{stst}).

$L$ is {\em geodesically conformally asymptotically flat}  if its
associated Randers space  $(S,R)$  is
  geodesically   asymptotically flat, that is, if $h$ and $\omega$ in
\eqref{f1} and \eqref{f2} satisfy \eqref{asymptflat} in some
asymptotic coordinates, for some $p, q>0$.

In this case, $L$ is {\em geodesically asymptotically flat} if
$\beta$ satisfies in asymptotic coordinates:
\bal\label{betasymptbis}
&\beta=1+O(1/|x|^{q'}),\quad\quad\partial_i\beta=O(1/|x|^{q'+1}),
\eal for some $q'> 0$.

\end{defn}
Consistently with Remark \ref{NewRemark1},
  this notion concerns truly geometric elements defined on the
spacelike section $S$ which are independent on the chosen standard
splitting: the norm of $Y$, i.e $\beta=-g(Y,Y)$, the  metric
$h=(g_0/\beta)+\omega^2$  and the cohomology class  of $\omega$ (the latter univocally
determined on the chosen $S$ by the one-form metrically associated
to  the orthogonal projection of $Y$ on $TS$).
In the particular case of static spacetimes,  it is irrelevant if the
spacelike hypersurface $S$  is chosen or not orthogonal to the
static vector field.

Let us now derive  some growth conditions on the metric coefficients of a given standard stationary splitting implying  geodesical  asymptotic flatness   in the sense of
Definition \ref{referee2def}.

\bt\label{g0} Consider a standard stationary spacetime
$(S\times\R, g)$  that satisfies in each end $E^{(k)}$:
\bal \label{asymptflat2}
&\big((g_0)_{ij}-\delta_{ij}\big)+|x|\partial_{l} (g_0)_{ij} =O(1/|x|^{p_0}),\nonumber\\
&(\omega_0)_j+|x|\partial_{l} (\omega_0)_j =   O(1/|x|^{q_0}),\\
&(\beta-1)+|x|\partial_{l}\beta=O(1/|x|^{q'}), \nonumber
\eal for some
$p_0, q_0, q'
>0$.
Then  $(S\times\R, g)$ is geodesically asymptotically flat
(with $p=\min\{p_0, 2q_0, q'\}$ and $q=\min\{q_0,q'\}$). \et
\begin{proof}
We recall that
\[h_{ij}=\frac{(g_0)_{ij}}{\beta} +   \frac{(\omega_0)_i(\omega_0)_j}{\beta^2}\quad\text{and}\quad \omega_i=\frac{(\omega_0)_i}{\beta};\]
hence,
\[h_{ij}=\frac{\delta_{ij}+O(1/|x|^{p_0})}{1+O(1/|x|^{q'})}+\frac{O(1/|x|^{2q_0})}{(1+O(1/|x|^{q'})^2}=\delta_{ij}+O(1/|x|^{\min\{p_0,2q_0,q'\}}).\]
Moreover,
\baln
 \lefteqn{\partial_kh_{ij}= } &  \\
  & -\frac{\partial_k\beta}{\beta^2} (g_0)_{ij}+\frac{1}{\beta}\partial_k(g_0)_{ij}-\frac{2\partial_k\beta}{\beta^3}   (\omega_0)_i(\omega_0)_j  +\frac{1}{\beta^2}\big(\partial_k(\omega_0)_i(\omega_0)_j+(\omega_0)_i\partial_k(\omega_0)_j\big)\\
&=O(1/|x|^{q'+1})+O(1/|x|^{p_0+1})+O(1/|x|^{q'+1+2q_0})+O(1/|x|^{2q_0+1})\\
&=O(1/|x|^{\min\{p_0,q'\}+1})+O(1/|x|^{2q_0+1}). \ealn 
Arguing analogously, we get:
\[\partial_k\omega_i=O(1/|x|^{q'+1})+O(1/|x|^{q_0+1})=O(1/|x|^{\min\{q_0,q'\}+1}).\]
Therefore, defining $p=\min\{p_0, 2q_0, q'\}$ and
$q=\min\{q_0,q'\}$, we see that $h$ and $\omega$ satisfies
\eqref{asymptflat} and then, since $\beta$ satisfies
\eqref{betasymptbis} by assumptions, $(L, g)$ is geodesically
asymptotically flat.
\end{proof}
\bere\label{referee2rem_moved}
  It is worth to observe that the addition
of further fall-off hypotheses for the second derivatives,  as in \eqref{asymptflat3} ,
would not be so innocent as it seems. Recall that such
fall-off hypotheses are commonly   used; in fact,  in order to define
asymptotic flatness by means of the approach based on Penrose
conformal embeddings and boundaries, one uses commonly the
simplifying hypothesis that the spacetime is vacuum (Ricci-flat)
in a neighborhood of the asymptotic boundary $\mathcal{J}^\pm$
(see for example \cite{Frau,HE,Town,wald})\footnote{Even though this
excludes the presence of electromagnetic fields, it is justified
as it  allows one to make simpler statements, leaving to the reader
the task to determine precise fall-off hypotheses for  the energy-momentum tensor (compare, e.g., with
\cite[Section
2.3]{Frau}) whenever the vacuum assumption is dropped.}.
 For a  vacuum  stationary solution  of the Einstein equations, there exist coordinates in which the metric is analytic
\cite{Mu70}. Moreover, for an asymptotically flat and vacuum  stationary solution there exists a system of coordinates in a neighborhood of infinity where  $h_{ij}-\delta_{ij}$, $\omega_i$, $\beta-1$ decay  as $1/|x|$ and each of their derivatives of order $k$,  as $1/|x|^{1+k}$, for any $k\in\N$ (see
\cite[Sect. 2.3]{Chru} and  \cite[Theorem 2.1, Eq. (26) and Lemma 2.5]{dain}).

Nevertheless,   surprising difficulties appear in the stationary
non-static case, because of the implications of the fall-off of
curvature. First, recall that very few examples of exact solutions
modeling {\em vacuum and rotating} isolated objects in general
relativity are presently known. The list of useful solutions
presently consists of the Neugebauer-Meinel dust (a rigidly
rotating thin disk of dust with finite radius surrounded by an
asymptotically flat  vacuum region), and a few variants, see
\cite{MMV1, MMV2}. Moreover, as emphasized by Roberts
\cite{Roberts}, {\em there is no known perfect fluid source which
can be matched to a Kerr vacuum exterior}, as one would expect in
order to create the simplest possible model of a rotating
star\footnote{See, about the difficulty of this problem, the claim
on its over-determinacy in \cite{MS}, as well as the construction
of a solution involving Kerr-Newman spacetime \cite{Mars} (recall
that the latter is  not vacuum around $\mathcal{J}^\pm$).}
---this  contrasts  with the plenitude of solutions which match to
Schwarzschild. So, the true applicability of the results in this
contexts requires accurate hypotheses, as the optimized fall-off
hypotheses in our definition.
\ere

Finally, it is also interesting to check further the consistency
of our notion of  geodesic {\em conformal} asymptotic flatness,
with the classical one obtained by using Penrose conformal
embeddings. On the one hand, in the (conformally) vacuum case the
classical notion of such asymptotic flatness (as explained, for
example, in \cite[Sect. 2.3]{Frau}) imply the existence of a
double cone structure (essentially, two copies of $S^2\times \R$)
for the points of the conformal boundary which are accessible from
the spacetime by means of causal curves. As emphasized in
\cite{FHSatmp}, there are reasons to prefer the recently revisited
notion of {\em causal boundary} to the conformal one, as the
former is a (explicitly intrinsic) general construction,
applicable even when no useful conformal embedding is known
---but, under quite general hypotheses,  it agrees with the conformal boundary when   this  can   be
defined. In the case of stationary spacetimes, the
stationary-to-Randers correspondence allows to characterize the
causal boundary as a double cone structure on the Busemann
boundary of the associated Randers manifold \cite{FHS}; in
particular, it agrees with Penrose's in the classical vacuum case
(see further details in \cite{FH}). On the other hand, in the
classical notion, the Weyl tensor goes to 0 on the boundary. As in
the case of the Riemann tensor, our mild fall-off requirements in
Defn. \ref{referee2def} are tailored for geodesics and do not
imply such a behavior ---but, obviously, our definition and
results are applicable if such an additional hypotheses is
imposed, as in Weyl conformally flat vacuum solutions
\cite[Chapter 20]{SKMHH}.

\subsection{Large spheres in the spacelike slice of an end}
Next, let us apply the results on convexity to  some   hypersurfaces
in  a geodesically  asymptotically flat spacetime.   As in Subsection~\ref{asflat}, we will consider  large spheres
$S^{n-1}(r_0)$ defined in   the  coordinates of each  end as $|x|^2=r_0^2$   ((regarded as the boundary of the region where $|x|^2<r_0^2$, except if otherwise specified).
As a direct
consequence of Proposition \ref{largesphere} and Theorem
\ref{lighteq}:

\begin{cor}\label{NewCorollary1}For any geodesically  conformally asymptotically flat
spacetime, all the hypersurfaces $S^{n-1}(r_0)\times \R$ are
  strongly   light-convex for any sphere  $S^{n-1}(r_0)$ of large
radius $r_0$ (i.e., with $r_0$ greater than some constant).
\end{cor}
However, time-convexity is subtler, as large balls will not
fulfill this property in all asymptotically flat spacetimes. In
fact, recall first the following straightforward consequence of
Theorem \ref{peq2} and Corollary \ref{betabehaviour}.

\begin{cor}\label{NewCorollary2} Let $L=  S \times \R $
be  a geodesically  asymptotically  flat standard
stationary spacetime. If, in addition to \eqref{betasymptbis},
$\beta$ satisfies
\beq\label{decreasing}
\partial_r\beta \sim-\frac{C}{|x|^{q'+1}}\eeq
   being\footnote{The result
follows even if we allow  here and in \eqref{betasymptbis}
$q'=0$.} $q'\in (0,2q)$  and $C>0$, then the hypersurfaces
$S^{n-1}(r_0)\times \R$ are   strongly   time-convex for any sphere
$S^{n-1}(r_0)$ of large radius $r_0$.
\end{cor}

\begin{remark} The additional condition on $\partial_r\beta$
in Corollary \ref{NewCorollary2} is crucial for time convexity,
 as  it  can be easily understood from Lorentz-Minkowski spacetime. Recall that the hypersurfaces $S^{n-1}(r_0)\times \R$   in $\mathbb{L}^n$ are time-convex, but not strongly time convex. In fact,  any line ${x_0}\times \R$ with $x_0\in S^{n-1}(r_0)$, regarded as a timelike geodesic, remains in the boundary of the domain and does not leave it. So, the sign of $\partial_r\beta$ will be crucial, as the next
corollary will make apparent. However,  \eqref{decreasing} is not
satisfied by physically reasonable asymptotically flat spacetimes,
as discussed below  (see Proposition~\ref{positivemass}).
\end{remark}

\bco\label{NewCorollary3}  Consider an asymptotically flat
spacetime  as in Theorem \ref{g0}; let $C>0$ be a constant and
let $S^{n-1}(r_0)$ be a sphere of radius $r_0$ in $E^{(k)}$. 
Then the following properties hold:
\begin{enumerate}
 \item if $q'\in (0, 2q_0)$ and
\[\partial_r\beta \sim-\frac{C}{|x|^{q'+1}},\quad \text{as $|x|\to +\infty$,}\]
 then  the hypersurfaces $S^{n-1}(r_0)\times \R$ are   strongly time-convex for
any $r_0$ large enough;
\item if
\[\partial_r\beta \sim\frac{C}{|x|^{q'+1}},\quad \text{as $|x|\to +\infty$,}\]
then, for any $r_0$ large enough, the hypersurfaces $S^{n-1}(r_0)\times \R$ are (strongly   light-convex  but) not  infinitesimally   time-convex.
\end{enumerate}
\eco
\begin{proof}
The first part follows  from  Theorem~\ref{g0}   and Corollary~\ref{NewCorollary2}.
In fact, if  $q'\in (0,q_0]$, we get  that $q=q'$ and  $q'<2q$, while for $q'\in (q_0, 2q_0)$, we have $q=q_0$ and then again $q'<2q$.
For the second part, we first observe that for a large enough $r_0$ the hypersurfaces
$S^{n-1}(r_0)\times \R$ is light-convex by   Theorem~\ref{g0} and   Corollary~\ref{NewCorollary1}. Moreover,
considering $\phi=\phi(r)=\ro^2-r^2$, we get
\baln
h(\nabla^h\phi, \nabla^h\beta)&= -2|x|\partial_r\beta +O(1/|x|^p)|x|\big(\sum_k (\partial_k\beta)^2\big)^{1/2}\\
&<-2|x|\frac{C}{2|x|^{q'+1}}+O(1/|x|^p)|x|O(1/|x|^{q+1})\\
&=-C/|x|^{q'}+O(1/|x|^{p+q'}),
\ealn
which is negative for  $r_0$ large enough.
Thus, Eq. \eqref{necessary2} is not
satisfied on $S^{n-1}(\ro)$ and Proposition \ref{Rbetaconvexity2} plus Theorem
\ref{peq2} implies that the hypersurfaces  $S^{n-1}(\ro)\times \R$ are
not time-convex for  any  $\ro$ large enough.
\end{proof}
\bere  Observe that, as a difference with Corollaries
\ref{NewCorollary1} and \ref{NewCorollary2},   where only
assumptions on the first derivatives of $\omega$ are considered,
in Theorem~\ref{g0} and in Corollary~\ref{NewCorollary3}, also
conditions  on the asymptotic behavior of the one-form $\omega_0$
are imposed (\eqref{asymptflat2} are indeed used). 
Nevertheless, such asymptotic conditions for $\omega_0$ are not
necessary if we consider an appropriate combination with the
behavior of $g_0$, in the spirit of Definition \ref{referee2def}.
Recall that if
one changes  the standard stationary splitting by considering a
slice $S'=\{(x, f(x))\}$ as in Remark~\ref{NewRemark1},  the new
standard stationary spacetime $(S'\times\R, g)$ remains, of
course, geodesically asymptotically flat  and large spheres in any
end $(E')^{(k)}=\{(x, f(x))|x\in E^{(k)}\}$, will produce
hypersurfaces in $S'\times\R$ time-convex or not according to
Corollary~\ref{NewCorollary3}. This is because, as already
recalled in Remark~\ref{NewRemark1}, the metric $h^f$ on $S'$ is
isometric to $h$, $\beta$ is invariant and the one-form $\omega^f$
on $S'$ is the push-forward, by the map $x\in S\mapsto (x,f(x))\in
S$, of $\omega-\de f$. \ere

In the light of Corollary~\ref{NewCorollary3}, it becomes important to give conditions
  in order to understand the sign   of the constant $C$ in the asymptotic behaviour of $\beta$.
Under natural physical hypotheses, this coefficient is the Komar's
mass and, so, its sign is expected to be non-negative \cite[p.
462]{Choque09}. More precisely, following Choquet-Bruhat
\cite[Defn. 9.1, p. 55]{Choque09}, a $4$-spacetime will be called
{\em Einsteinian} when it satisfies Einstein equations with
reasonable conditions for the matter. This definition is
consciously ambiguous. For an asymptotically flat spacetime, here
{\em Einsteinian} will mean just: (a) it fulfills
Definition~\ref{referee2def} with $p=q=q'=1$ (i.e., the natural
power for the asymptotic decay holds), (b) either $g_0$ is
complete or the Komar mass in a $r$-bounded region  that includes
the removed compact subset $K$ is nonnegative (i.e., at each end,
the integral of Komar form $\star d Y^\flat$ on a compact surface
enclosing $K$ is nonnegative) and, (c) in the particular case that
Komar mass of an end is $0$, positive mass theorem is applicable
(recall that Komar and ADM masses are expected to be equal
\cite[Th. 4.13]{Choque09}), so that the spacetime is
$\mathbb{L}^4$. Other natural physical  assumptions will be stated
explicitly in the following result,  which follows from
\cite[Theorem 4.11, p.462]{Choque09} and  Corollary \ref{NewCorollary3}.

\bpr\label{positivemass}
  Let $(S\times\R,g)\neq \mathbb L^4$ be an  
Einsteinian stationary spacetime.   Assume that,  at each end,  there exists a constant $C$ such that
\[\partial_r\beta=C/|x|^{2}+ o(1/|x|^2), \quad\quad\text{as $|x|\to +\infty$}.\]
Moreover, assume that the Ricci tensor $\mathrm{Ric}(g)$ of the metric  $g$ satisfies
\[\mathrm{Ric}(g)(\partial t,\partial t)\geq 0,\]
and   the function   $\mathrm{Ric}(g)(\partial t,\partial t)$  is
integrable on $S$. Then,  at  each end, $C$ must be positive
and,  for  any $r$ bigger than some large $r_0>0$, the
hypersurface $S^{n-1}(r)\times\R$ is not infinitesimally
time-convex.  \epr \bere In the next subsection, we will study
Kerr spacetime as a paradigmatic example and will check
directly the behavior of $\beta$ ensured by
Proposition~\ref{positivemass}. But such  an asymptotic behaviour
holds for any physically reasonable asymptotically flat spacetime,
as discussed above (see also \cite[\S 3.1]{BeiSch00}). \ere
\subsection{Convex shells in  Kerr spacetime}\label{kerrshell}
  We will focus now on Kerr spacetime.  We will check that, even
though large asymptotic spheres   are strongly light-convex, {\em they are
not infinitesimally time-convex},   and take advantage of the  intuition on this
spacetime in order to interpret physically such a result.

Recall that in Boyer-Lindquist coordinates $(r,\theta,\varphi,t)$
the metric of the Kerr spacetime is given by
\[\de s^2=-\frac{\Delta}{\rho^2}\left(dt-a\sin^2 \theta d\varphi\right)^2+\frac{\sin^2\theta}{\rho^2}\left((r^2+a^2)d\varphi-a dt\right)^2+\frac{\rho^2}{\Delta}dr^2+\rho^2d\theta^2,\]
\baln \Delta= r^2-2mr+a^2,  &&\rho^2=r^2+a^2\cos^2\theta, \ealn
where $m$ is the ADM mass and  $a=j/m$ with $j$ the ADM
angular momentum of  the spacetime. The above metric  is
standard stationary outside the stationary  limit  hypersurface   $\mathcal N$
({\em ergosurface}
$r=m+\sqrt{m^2-a^2\cos^2\theta}$),  which
appears in {\em slow } Kerr spacetime (i.e., the black hole model
where $a^2<m^2$). So, by Theorem~\ref{g0}, it is immediate to
check that this region is a  geodesically  asymptotically
flat stationary spacetime; thus, Proposition \ref{largesphere} is
applicable and from Theorem~\ref{lighteq},  the submanifolds
$S^2(\ro)\times \R$
are strongly light-convex for $\ro$ large enough. For
time-convexity, recall that, from the expressions above,
\[\beta=
\frac{r^2-2mr +a^2\cos^2\theta}{r^2+a^2\cos^2\theta}=1-\frac{2mr
}{r^2+a^2\cos^2\theta},  \] 
Thus, (2) of
Corollary~\ref{NewCorollary3} holds and the submanifolds
$S^2(\ro)\times \R$ are not time-convex for $\ro$ large enough.

The same results hold also in the stationary region of the
Kerr-Newman spacetime. In fact, the metric of the Kerr-Newman
solution is obtained from the Kerr one by replacing $\Delta$ with
$\Delta+q^2$, where $q$ is the electric charge of the
spacetime. Therefore,  it is geodesically  asymptotically flat
and, since
\[\beta= 1-\frac{2mr
-q^2}{r^2+a^2\cos^2\theta},\]
(2) of Corollary~\ref{NewCorollary3} is satisfied as well.

\begin{remark}\label{fall}
These results admit the following natural physical
interpretation. First, let us consider the non-time convexity of
the large asymptotic balls. Assume that a pebble is tossed
straight upward. If it lacks escape energy, the gravitational
attraction will make the pebble to reach a maximum value $\ro$ and
fall down. The trajectory of the pebble will be then a timelike
geodesic which violates the local time-convexity of the sphere of
radius $\ro$. More precisely, any timelike geodesic in
Kerr(-Newman)  spacetime with a $r$-turning point $r_0=r(s_0)$
which is a strict local maximum of\footnote{Following the detailed
discussion and nomenclature   about Kerr in
\cite[p.209]{oneikerr}, such geodesics are either (ordinary) {\em
bounded orbits} or (exceptional) {\em crash-crash} ones. Recall
that  the ordinary orbits of Kerr geodesics are  {\em bounded}
(those with two $r$-turning points, one of them a maximum), {\em
flyby} (one $r$-turning point, necessarily a minimum) or {\em
transit} (no $r$-turning points). The exceptional orbits are {\em
spherical} ones (i.e. $r(s)\equiv r_0$), asymptotic orbits to a
spherical one, {\em crash-escape} orbits and {\em crash-crash}
ones. None of these have
 a $r$-turning point $r_0$ except the crash-crash one (as in the example of
 the pebble), where $r_0$ is a maximum}  $r(s)$, violates the
time-convexity of $S^2(r_0)\times\R$.  However, it is easy to
realize that such a behavior does not hold if, instead of a
pebble, a light-ray is emitted (as light-rays initially
propagating both, radially outwards and tangent to a large sphere,
will attain increasing arbitrarily big values of $r_0$), so that
light-convexity is always achieved.


\end{remark}
As a final digression,  we emphasize that, in the computations
above, only   the stationary region  of the spacetime is being
considered. For example, in the Kerr spacetime, such a region lies
outside the  ergosurface $\mathcal{N}$.  In
this hypersurface the Killing vector $\partial_t$ is lightlike
($\beta=0$) and the induced metric on $\mathcal{N}$ is degenerate
at the poles $\theta=0,\pi$. So, it cannot be regarded as a
(timelike) boundary $\partial D\times \R$ of a domain in the same
sense as above. Even though the notions of time and light
convexity would make still sense, new subtler possibilities would
appear. The final applications on connecting geodesics are
expected to hold with independence of the details of this boundary
(see Remark \ref{referee2rem2}). Nevertheless,  for the sake of  completeness, we
make some computations about the convexity of hypersurfaces close
to $\mathcal{N}$. Following \cite[Sect. 7.2]{mas}, we consider the
domain $M^a_\eps:=D_\eps^a\times \R$, for each $\eps>0$, where
\be \label{eda}D^a_{\eps}=\{(x,y,z)\in \R^3 \colon
m+\sqrt{m^2+\eps-a^2\cos^2\theta}<r\}.\ee
These domains $M^a_\eps$ exhaust  $M^a := M^a_{\eps =0}$
 when
$\eps\searrow 0$ and, in the case of Schwarzs\-child spacetime
(i.e. $a=0$), they are just spheres  with radius
$r_\epsilon=2m+\eps$ (greater than Schwarzschild radius $r=2m$).
Observe that in the Schwarzs\-child spacetime, being $\omega\equiv
0$, the Fermat metric reduces to the Riemannian metric
$h=(\frac{1}{\beta}dr^2+r^2\de \theta^2+r^2\sin^2\theta\de
\varphi^2)\frac{1}{\beta}$ where now $(\beta\equiv)
\beta(r)=1-2m/r$. The function  $\phi_\epsilon$,  defining
the sphere of radius $r_\epsilon$ as its $0$-level set and
assuming positive values on the domain $D_\epsilon=\{(x,y,z)\in
\R^3 \colon r_\epsilon<r\}$, is
$\phi_\epsilon(r)=r^2-r_\epsilon^2$.  Thus the gradient (with
respect to $h$) of $\beta$ does not point outside $D_\epsilon$
(i.e. \eqref{necessary2} is satisfied) and \eqref{tcon} becomes
equivalent to the infinitesimal convexity of the sphere $\partial
D_\epsilon$ with respect to $h$.  For Schwarzs\-child
spacetime,   the Christoffel symbols of $h$ involved in the
computation of  $H^h_{\phi_\epsilon}$  on $T\partial
D_\epsilon$ are: \baln \Gamma^r_{\theta\theta}=m
-r_\epsilon\beta,&& \Gamma^r_{\theta\varphi}=0,&&
\Gamma^r_{\varphi\varphi}=(m-r_\epsilon\beta)\sin^2\theta \ealn
hence, from Theorems \ref{lighteq} and \ref{peq2}, $\partial
D_\epsilon\times\R$ is   strongly  light-convex  and strongly
time-convex convex if $m-r_\epsilon\beta   >0 $, i.e. if
$0<\epsilon  < m $ (observe that strong light-convexity seems to
fail at $\theta=0,\pi$ but these values must not to be taken into
account because belong to the boundary of the open subset where
spherical coordinates are injective; clearly by rotating the $r$
semi-axis, also the poles on the sphere of radius $r_\epsilon$
corresponding to $\theta=0,\pi$ can be covered). For $\eps=m$, we
have $m-r_\epsilon\beta=0$ and then $\partial
D_{\epsilon=m} \times\R$ is strongly time-convex but only
light-convex. For arbitrary values of $a$, the shape of the
  ergosurface  changes and such convexities are not expected. However,
    for $\eps$ and $|a|$ sufficiently small,
Kerr regions  $M^a_\eps$ are
strongly   time- and   light-convex too. A simple  proof of this fact (simplifying and
slightly improving the one  in   \cite[Prop. 7.2.1]{mas})  follows  recalling that,  in this case,
 $M^a_\eps$ can be regarded as a small perturbation of   $M^{a=0}_\eps$  and Kerr metric as a small
$C^2$ perturbation of   Schwarzs\-child one,  when $a$ goes to $0$.  By the strong  light-convexity
of  $M^{a=0}_\eps$, any
of  \eqref{hes}-\eqref{hes3}  is satisfied   with a  strict inequality. As $\omega \equiv 0$ and  $
\partial D^{a=0}_\eps$ is compact,  the $h$-Hessian of $\phi_\eps $ is less than a strictly negative
constant, thus the same must   hold for  $\partial D^a_\eps$  in the  Kerr spacetime,
provided that  $a$ is small enough
(to this end, notice that in the Kerr metric  $\omega  = - (2mra \sin^2 \theta d \varphi )/ (\rho^2  -2mr)$  
(see, for example, \cite{GiHeWW09})
and its
derivatives goes to $0$ when $a$ tends to $0$, uniformly on $\partial D_\eps$). For strong
time-convexity we can reason in a similar way,   because, as already observed   above,  \eqref{tcon} holds with strict inequality    on the boundary of Schwarzs\-child spheres.
 Nevertheless, space-convexity does not
hold even for  small $|a|$ (including $a=0$) and small
 $\epsilon>0$, see \cite[Corollary 2]{FlSa}.

Putting together the previous   discussion    plus  Corollary~\ref{NewCorollary1},   we get  the following result  about  light-convexity   of some shells in  
slow Kerr   spacetime:
\begin{cor}\label{cKerrlc}  Let $m > 0$   and $a\in\R$  such that  $a^2 < m^2$.
For any   $\eps\in (0,m)$,   and $\ro\in (0,\infty]$, let
$D^a_{\eps,\ro}$ be the following subset of $\R^3$:
\[
D^a_{\eps,\ro}=\{(x,y,z)\in \R^3 \colon
m+\sqrt{m^2+\eps-a^2\cos^2\theta}<r<\ro\}.\]
Then   there exists $a_0>0$, depending on $\eps$, such that the  stationary
domain $M^a_{\eps,\ro}:= D^a_{\eps,\ro}\times\R $ of Kerr
spacetime $M^a$ has  strongly   light-convex   boundary for each  $|a|<a_0$,
provided that $ \ro $ is large enough.
\end{cor}
\section{Applications to topological lensing and causal simplicity}\label{s4}
\subsection{Lightlike geodesics}\label{s4.1}
In  \cite[Proposition 4.1]{cym},  using Proposition \ref{fp}, it  was
proved that  if $(S, F)$ is
forward or backward complete then any point $w =(p,t_p)$ of $L$ can be
joined to a flow line $l= l(\tau)=(q,\tau)$ of the Killing field $\partial_t$ by
means of a  past--pointing  lightlike geodesic $\gamma$ and,
moreover, if $S$ is non--contractible,  a sequence $\{\gamma_m\}$
of such geodesics with  negatively   diverging arrival times exists.
Clearly, from Proposition \ref{fp}, the existence of at least one such a geodesic  is an immediate
consequence of the completeness assumption on $(S,F)$ and of the Hopf-Rinow theorem in Finsler
geometry (cf. \cite[Theorem 6.6.1]{bcs}). From an
infinite-dimensional variational viewpoint, the multiplicity result is obtained applying
L\"usternik--Schnirelmann theory to the energy functional of the Finsler metric. Indeed, if $S$ is
non--contractible, the L\"usternik--Schnirelmann category of the manifold $\Omega_{p,q}(S)$ of
$H^1$-paths between  the points $p$ and $q$  on $S$ is infinite, \cite{fh}, and then the energy functional, which
satisfies the Palais--Smale
condition on $\Omega_{p,q}(S)$, \cite{cym}, must admits infinitely many critical points,  whose energy goes to infinity, and then
infinitely many geodesics  between $p$ and $q$. If $p$ and $q$ lie on the same closed
geodesic,  the  geodesics connecting $p$ to $q$ might  be  the multiple coverings of  the closed one,
as happens for two non-antipodal points on a sphere; on the other hand, even  in this case, the
lightlike geodesics on $S\times\R$ arising from the Finsler ones on $S$ have distinct supports
(while the projections on $S$ of the supports coincide). Notice that this multiplicity result admits
a natural interpretation of {\em topological lensing} of the connecting geodesics. In fact, light
rays emitted at different moments in the past from the stellar object,   represented by $l$,  
will arrive at different directions for the observer at  the single event $w$. 

 In \cite{cys}, the  above mentioned result in \cite{cym}
was reinterpreted and improved. Concretely,  the
completeness  of $(S, F)$ was substituted by the weaker assumption
that  the closed balls $\bar B_s(p,r)$ are compact. In fact, this
condition becomes equivalent to  the global hyperbolicity of $L$
(recall Theorem \ref{tscr}(2)) and, in such a spacetime, any two
causally related points can be joined by means of a
length-maximizing causal geodesic (Avez-Seifert result). As a
consequence, the existence of the arrival-time  maximizing, past--pointing  
lightlike geodesic $\gamma$ follows directly from purely causal
grounds (just realizing that, because of stationarity, $ J^-(p) $
must intersect any  flow line $l$ of $Y$). The variational technique is
required only for the existence of the sequence $\{\gamma_m\}$
and, as also shown in \cite{cys},  the assumption on $\bar
B_s(p,r)$ is enough for this purpose.
Recall that, the existence and multiplicity of connecting light rays is then equivalent to the
existence and multiplicity of geodesics connecting two points in a Randers manifold. The latter
problem is well understood in Finsler geometry (see for example \cite{bcs}) and can appear
because of topological reasons (as explained above) or just by curvature focalization (existence of
conjugate points for the Randers manifold). So, the existence of gravitational lensing is completely
characterized from this viewpoint.

Next, we can extend such results to domains $D$ of $S$ having
convex boundary.  This will be a consequence of the previous
results combined with the results in \cite{bcgs} where, first, the
equivalence between different notions of convexity for the
boundary $\partial D$ of a domain $D$  in a Finsler manifold    was proven and, then, the
convexity of $D$ is characterized as follows (see \cite[Theorem
1.3]{bcgs}):
\begin{teo}\label{main}
Let $D$ be a  $C^2$   domain\footnote{Recall
that \cite[Theorem 1.3]{bcgs} is stated under $C^{2,1}_{\rm loc}$
regularity, but the result holds also for a $C^2$ domain (see \cite{caponiogelogra}).}
of a smooth manifold $M$ endowed with a Finsler
metric $F$ having $C^{1,1}_{\rm loc}$ fundamental tensor (see
\eqref{fundmatr}) and such that the intersection of the closed
symmetrized balls $\bar B_s(p,r)$, $p\in D$,
with $\bar D$ is compact. Then, $D$ is convex if and only if
$\partial D$ is  convex.

Moreover, in this case,  if additionally   $D$ is not
contractible, then any pair of points in $D$ can be joined by
infinitely many connecting geodesics contained  in $D$ and having
diverging lengths.
\end{teo}
\begin{remark}\label{rapostilla}
From the technical viewpoint, an improvement of the previous
result will become relevant later. Recall that we can consider
also the distance $d_s^{\bar D}$ obtained by computing the
distances (the Finslerian distance $d$ and then the symmetrized
one $d_s$) by using only curves contained entirely  in $\bar D$.
This distance is intrinsic to $\bar D$, i.e., it is independent of
the extension of $F$ outside. Obviously, $d_s^{\bar D}$ is greater
or equal to the restriction of $d_s$ to $\bar D$ and, so, the
corresponding balls satisfy $\bar B_s^{\bar D} (p,r) \subset \bar
B_s(p,r) \cap \bar D$. Then, if $\bar B_s(p,r)$ is compact so is
$\bar B_s^{\bar D}(p,r)$, but the converse does not hold (see
Example \ref{noncompact} below). By checking the proof of Theorem
\ref{main} in \cite{bcgs}, it is not hard to prove:
\begin{quote} {\em All the conclusions of Theorem \ref{main} hold if
the hypothesis on the compactness for $\bar B_s(p,r)$, is replaced
by the more general hypothesis of compactness for the intrinsic
balls $\bar B_s^{\bar D}(p,r)$, $p\in D$.}
\end{quote}
In fact, the compactness of the sets $\bar D\cap \bar B_s(p,r)$ is
used in Lemma 4.2 of \cite{bcgs} and in the proof of the
Palais-Smale condition for the functionals of a family perturbing
the energy functional of the Finsler manifold (see \cite[Prop.
4.3]{bcgs}), to ensure that the supports of any sequence of
absolutely continuous curves $\gamma_n\colon [0,1]\to D$,
$\gamma_n(0)=p$, $\gamma_n(1)=q$, $p, q\in D$, such that
\[\int_0^1 F^2(\dot\gamma_n)\de s\leq
C,\]
for a constant $C>0$ independent of $n$,
are contained in a compact subset of $M$. The same holds if the intrinsic balls $\bar B_s^{\bar D}(p,r)$, $p\in D$,
are compact. In fact
\[d^{\bar D}(p,\gamma_n(s))\leq
\int_0^sF(\dot\gamma_n)\de\tau\leq
\left( \int_0^1 F^2(\dot\gamma_n)\de s \right)^{\frac 1 2}\leq C^{1/2}
\]
and analogously $d^{\bar D}(\gamma_n(s),q)\leq C^{1/2}$, hence the
supports of the curves $\gamma_n$ are contained in the subset
$\bar B^{\bar D}_s(p,C^{1/2}+\frac{d^{\bar D}(q,p)}{2})$ of $\bar
D$. \end{remark}

Our main result for lightlike geodesics is then:
\begin{teo} \label{appl1}
Let $L=S\times\R$ be a  $C^2$
spacetime endowed with a $C^{1,1}_{\rm loc}$
 standard stationary  Lorentzian metric $g_L$, and let    $D$ be a  $C^2$
domain of   $S$. Consider on $S$ the Fermat metric $F$
defined in \eqref{efermat} (see also \eqref{ran}, \eqref{f1} and
\eqref{f2}), and assume that any intrinsic closed symmetrized ball
$\bar B_s^{\bar D}(p,r)$, $p\in D$, is compact (which happens, in
particular,  if any   $\bar B_s(
 p,r)  \cap\bar D$ is compact). Then, the following assertions
are equivalent:
\begin{enumerate}

 \item $(D\times\R,g_L)$ is causally simple.

 \item $(D,F)$ is convex.

 \item $(\partial D ; F)$ is convex  (infinitesimally or, equivalently, locally,  \cite[Corollary 1.2]{bcgs}).

 \item $(\partial D \times \R; g_L)$ is light-convex  (infinitesimally or, equivalently, locally, recall Corollary~\ref{lightequiv}).

\item Any point $w  = (p , t_p ) \in D \times \R$ and any line
$l_q:=\{ ( q,\tau ) \in D \times \R: \tau \in \R\}$, with $p\neq
q$, can be joined in $D\times\R$ by means of a future--pointing
lightlike geodesic $z (s) = (x (s),  t (s))$ which minimizes the
(future) arrival time $T$ in (\ref{ar}) (equivalently, such that
$x$ minimizes the $F$-distance in $D$ between  $p$ and $q$).

\item Any $w$ and $l_q$ as above can be joined in $D\times\R$ by
means of a past--pointing lightlike geodesic $z (s) = (\tilde x
(s),  \tilde t (s))$ which  maximizes   the (past) arrival time $\tilde T$ in
(\ref{ar}) (equivalently, such that $\tilde x$ minimizes the
reverse $F$-distance in $D$ between  $p$ and $q$).
\end{enumerate}
 In this case, if $D$  is not contractible, then
sequences of both,  future--pointing and  past--pointing lightlike
geodesics joining $w=(p,t_p)$ and $l_q$ contained  in $D\times\R$
and with diverging arrival times,   exist  for all $p, q\in D$.
\end{teo}

\begin{proof}
For (1) $\Leftrightarrow$ (2) apply Theorem \ref{tscr}(1); for (2)
$\Leftrightarrow$ (3), Theorem \ref{main}  (plus Remark~
\ref{rapostilla});  for (3) $\Leftrightarrow$ (4)  Theorem
\ref{lighteq}. Finally,
 (5) (or (6))  $\Leftrightarrow$ (2) follows from Proposition~\ref{fp},
recalling also \eqref{ar} (notice also that (5) $\Leftrightarrow$
(1) holds in a more general context, see Remark \ref{rrr}(2)
below). The multiplicity result is a direct consequence of the
last statement in  Theorem \ref{main}.
\end{proof}

\begin{remark}\label{rrr} (1)  In other
results on the connection of $w$ and $l_q$ with a lightlike
geodesic \cite{fgm}, both, the light-convexity of $\partial D
\times \R$ and some  global assumptions on the growth of $\omega$ and $\beta$ in a standard
stationary splitting are required as sufficient conditions. In
comparison, our result is optimal because: (a) it characterizes
the light-convexity of $\partial D \times \R$ in terms of the
convexity of $(D,F)$, and (b) it gives the natural global
hypotheses (i.e., the compactness of  $\bar B_s^{\bar D}(p,r)$
--or at least  of $\bar B_s(p,r)\cap\bar D
 $)  which,  as it can be easily showed,
is always implied by the global assumptions in the
previous results.

The ambient hypothesis (compactness of  $\bar B_s^{\bar D}(p,r)$
in (b)) is not strictly necessary, but it appears clearly as a
natural hypothesis ---not just as a technical assumption.  In
fact, even for a Riemannian manifold, if the boundary of a domain
$D$ is supposed to determine the convexity of the domain, one
assumes typically that the domain is included in a complete
Riemannian manifold. This hypothesis can be weakened by the more
intrinsic assumption  that the   closed balls in $\bar D$ are
compact. Nevertheless, no more true generality is obtained then:
when the last hypothesis  holds  (either in the version that the
closed balls of $\bar D$ are intrinsic or in the one that they are
the intersection of balls in $S$ with $\bar D$ are compact), one
can modify the original Riemannian metric outside $\bar D$ so that
it becomes complete\footnote{One
could   try to improve this hypothesis by replacing it
with some condition intrinsic to $D$ (i.e, independent of if it is
regarded as a domain of a bigger manifold) or relaxing the
hypotheses of smoothness on $\partial D$. Nevertheless, such
conditions are rather technical even in the Riemannian case
\cite{bgs} (and, thus, even in the simple static case). Moreover,
the non-trivial relation between the symmetrized distance $d_s$
and the generalized distance $d$ in a Finsler manifold,
complicates more the situation  ---recall, for example, that $d_s$
may not be a length  metric,  which must be taken into account in the
picture of the intrinsic Cauchy boundary, see \cite[Remark
3.23]{FHS}. }. In the stationary/Finsler case, however, we
retain explicitly the weakest hypothesis, due to the subtleties of
the symmetrized distance (plus the  existence of other elements in
the full spacetime).

Our optimal result was known for the special case of standard
static spacetimes ($\omega_0 \equiv 0$ in (\ref{stst})) which is
simpler, as the associated Randers space becomes indeed a
Riemannian one (see \cite[Proposition 5.1]{bgs1}).

(2) The appearance of causal simplicity in Theorem \ref{appl1} can
be seen from the following general viewpoint. Consider any
causally simple spacetime $L$, a point $p\in L$ and an inextensible
future-directed causal curve $l$, which enters in $J^+(p)$ at some
point $q=l(t_0)$. Necessarily, $q\in \overline{J^+(p)}\setminus
I^+(p)$ and, as $J^+(p)$ is closed, $q$ belongs to the horizon
$J^+(p)\setminus I^+(p)$. Then, necessarily there exists a
future-directed lightlike geodesic $\gamma$ from $p$ to $q
(=l(t_0))$. Moreover, $\gamma$ minimizes the ``proper arrival
time''  of the observer whose world-line is a reparametrization of $l$,  in the sense that $t_0$ is the minimum value of the
parameter $t$ of $l$ such that $p$ can be connected with $l(t)$ by
means of a future-directed causal curve. In fact, a causal
spacetime $L$ is causally simple if and only if for any point $p$
and any future-directed (resp. past-directed)  causal
 curve  
$l$ which
enters $J^+(p)$ (resp $J^-(p)$), there exists a first point
$l(t_0)$ such that $l(t_0)\in  J^+(p) $ (resp. $l(t_0)\in  J^-(p) $)
---which necessarily can be joined with $p$ by means of a
lightlike geodesic.
\end{remark}
As a direct consequence of Theorem~\ref{appl1} and Corollary
\ref{cKerrlc}, we get:
\begin{cor}\label{corollario}
The regions  $M^a_{\eps,\ro}$ of  Kerr spacetime (see Corollary
\ref{cKerrlc}) are  causally simple  provided that  that $r_0$ is large enough and, for each $\epsilon$, $|a|\geq 0$ is small enough.
Thus, in this
case, for any point $ w=(p,t_p)$ and any line $l_q(\tau) = (q,  \tau) $  in
$M^a_{\eps,\ro}$,
both, a sequence of future--pointing and one  of past--pointing
lightlike geodesics connecting $w$ and $l_q$ with diverging
arrival times, exist.
\end{cor}
Existence and multiplicity results as the ones in
Corollary~\ref{corollario}  were already known for a region of the
type $M^a_{\epsilon,\infty}$ \cite{fgm,mas}.   They can be also
obtained  for some shells strictly contained in  the domain of
outer communication   and  intersecting the ergosphere, in a slow or
extreme Kerr-Newman spacetime, \cite{hp}.
\begin{remark}\label{referee2rem2}
For the physical applications of the results along all
Section \ref{s4}, the following observations are in order. First,
when the spacetime $L$ is globally hyperbolic, the technical
assumption that the intrinsic symmetrized closed balls $\bar
B_s^{\bar D}(p,r)$ are compact, is automatically satisfied for any
domain $D$   (recall Theorem~\ref{tscr} (2) and Remark~\ref{rapostilla})
Whenever $L$ models a physical isolated body with no  black holes,
$L$ is assumed to be globally hyperbolic and  geodesically asymptotically flat  and, thus, our results are applicable
---i.e.,  the existence of connecting geodesics inside large balls is
assured, and criteria to estimate their minimum  radius are also
provided. In the case that the spacetime  contains   a black hole,
$L$ would represent only its stationary part and the applicability
of the results to large balls depends also on the (time, light)
convexity of the stationary limit hypersurface $\mathcal{N}$,
where the causal character of $Y$ changes. If $L$   was   all the
domain of outer communication\footnote{We use standard terminology
as in \cite{Chru}.} (as happens in Schwarzschild spacetime)  $L$
is expected to be globally hyperbolic and all the results would be
applicable again. Otherwise (as happens in slow Kerr, where
$\mathcal{N}$ is the limit of the ergosphere),  the results may be
still applicable in some cases, as shown in Corollary
\ref{corollario}. And, at any case, the domain of outer
communication would include $L$ and would have as a boundary some
Killing horizon $H$ (by Hawking's theorem, under some technical
hypotheses, see \cite[Sect. 3.1.1]{Chru} and references therein).
$H$ would play the role of a light (and time) convex boundary, as
it is a degenerate hypersurface foliated by lightlike geodesics.
Therefore, the results on connecting geodesics are expected to
hold (even though such geodesics might eventually cross the
ergosphere for $Y=\partial_t$,   as in \cite{hp}).
\end{remark}

\subsection{Timelike geodesics}\label{s4.2} A further application  of
 Theorem \ref{main},
concerns timelike geodesics in standard stationary spacetimes.

Consistently
with the notations of Subsection~\ref{timeconvexity}, let
$d^{F_\beta}$ be the generalized distance
 on $\R_u\times S$ determined by $F_\beta=\sqrt{h_\beta}+\omega_1$  and $d_s^{F_\beta}$ be the
 corresponding symmetrized distance. For  each $(u_0,p_0)\in \R_u\times S$, let $\bar B_s((u_0, p_0),
r)$ denote the closure of the $r$-ball in $\R_u\times S$ with
respect  $d_s^{F_\beta}$. As before, $\bar B_s(p_0, r)$ denote the
closure of the $r$-ball in $S$  with respect to the  symmetrized
distance $d_s^{F}$ associated to the generalized distance $d^{F}$
on $S$ (given as $F= \sqrt{h} + \omega$),  and   $\bar B_s^{\bar
D}(p,r)$ denotes the intrinsic $r$-ball in $\bar D$ obtained from
the intrinsic distance   $d_s^{\bar D}$ (Remark \ref{rapostilla}).

Our aim is to prove the following result:

\begin{teo}  \label{appl2}
Let $L=S\times\R$ be a  $C^2$ spacetime
endowed with a $C^{1,1}_{\rm loc}$  standard stationary  Lorentzian metric $g_L$ and let
$D$ be a  $C^2$
domain of   $S$.   Assume that
all the intrinsic closed symmetrized balls $\bar B_s^{\bar
D}(p,r)$, $p\in D$ are compact (which happens, in particular,  if
all $\bar B_s(p,r)  \cap\bar D$ are compact). Then:

\begin{itemize}
 \item[(A)]   $(\R_u\times \partial D; F_{\beta})$ is convex (infinitesimally or, equivalently, locally, \cite[Theorem 1.1]{bcgs})
 if and only if for any $  \ell   > 0$, any point $w  = ( p, t_p )
\in D \times \R$ and any  line  $l_q(\tau) = (q,  \tau) \in D
\times \R$, $\tau \in \R$,
can be
joined by a
future--pointing timelike geodesic   $z (s) = ( x (s), t (s)) \in
D \times \R$, $s\in I=[0,1]$, with Lorentzian length $ \ell  $ such
that $x$
minimizes the arrival time $t(1)$ among all the future-pointing
causal curves from $p$ to $q$ of the same length $ \ell $.

\item[(B)]  If $(\R_u\times\partial D; F_\beta)$ is convex   and
$D$ is not contractible, then a sequence  of future--pointing
timelike geodesics $z_n(s) = ( x_n (s),  t_n (s))$,  $s\in
I=[0,1]$, having  Lorentzian length $\ell  $,  joining $w$ and
$l(\R)$,  having support in $D\times\R$ and   diverging arrival
times \eqref{ar1} exists.
\end{itemize}
\end{teo}
Before proving this theorem, the following comments are in order.

\begin{remark}\label{rappl2}
(1) From the technical viewpoint, it is worth pointing out that
the compactness of the subsets $\bar B_s(\bar p_0, r_0)\cap \bar
D, p_0\in D, r_0>0$, does not imply the compactness of the
analogous subsets $\bar B_s((u_0,p_0),  r  )\cap (\R_u\times\bar
D)$ in $\R_u\times S$ (even in the case when $\omega=0$), see
Example \ref{noncompact} below. Nevertheless, if we consider
instead the intrinsic balls, the analogous property will hold
(Lemma \ref{lcompact} below). Because of this reason, the
improvement in Remark \ref{rapostilla} of Theorem \ref{main}
becomes important here.

\smallskip

(2) The idea of the proof will be to reduce the problem to the
lightlike case, by using the metric $g_{L_1}$ explained in
Subsection \ref{timeconvexity} and taking into account the
previous point (1). So, the assertions in (A) can be
also refined taking into
account Theorem \ref{appl1}. In particular,
  analogous statements hold for past--pointing timelike
geodesics and    the hypothesis that  $(\R_u\times \partial
D; F_{\beta})$ is convex in (A) can be replaced by any of the
following alternatives:
\begin{enumerate}
\item[(i)] $( \R_u\times D ; F_\beta)$ is
 convex  (recall Theorem~\ref{main});
\item[(ii)] the boundary of $(\R_u\times D\times\R,g_{L_1})$ is
light-convex (infinitesimally or, equivalently, locally, Corollary~\ref{lightequiv});
 \item [(iii)] $(\partial D\times\R; g_{L})$  is time-convex  (infinitesimally or, equivalently, locally, recall Remark~\ref{timeequiv});
\item[(iv)] $(\R_u\times D\times\R,g_{L_1})$ is causally simple  (Theorem~\ref{appl1}, applied to the standard stationary spacetime $(\R_u\times D\times\R,g_{L_1})$).
\end{enumerate}
Moreover,   the  convexity of $(\R_u\times
\partial D; F_{\beta})$ is directly computable from Prop.
\ref{Rbetaconvexity2}.

\smallskip

(3) As a consequence, the causal simplicity of $(\R_u\times
D\times\R,g_{L_1})$ implies the causal simplicity of
$(D\times\R,g_{L})$;   notice that   the converse does not hold. Such a
property is general  for any spacetime $(L,g)$ (easily, if $(\R_u
\times L, g_1=du^2+g)$ is causally simple then so is $(L,g)$, see
\cite[Th. 3.6]{minguzzi} for a more general result). In the
stationary case, this is parallel to the fact that the
time-convexity for a boundary implies light-convexity, but the
converse does not hold.

\end{remark}

\begin{example}\label{noncompact}
Consider the standard static spacetime with $S=\R$ and
$\beta=e^{x^2}$, i.e., $L=\R^2, g_L=-e^{x^2}dt^2+dx^2$ so that the
Fermat metric is $F=\sqrt{h}$ with $h=dx^2/e^{x^2}$ and the balls
for $F$ and $h$ agree. Now,
$$\begin{array}{c} (L_1,g_{L_1})=(\R_u\times \R^2, g_{L_1}=du^2+dx^2-e^{x^2}dt^2),\\
\R_u\times S=\R^2, \quad h_1=e^{-x^2}(du^2+dx^2), \quad
F_\beta=\sqrt{h_1}.
\end{array}$$
Notice that $(\R^2,h_1)$ has a finite diameter $R_0(<\infty)$
(say, $d^{h_1}\left((0,0),(u_M,0)\right)<
 1+  2\int_0^\infty
e^{-x^2/2}dx$ for all  $u_M\in\R$, as one can take curves starting
at $(0,0)$ which go far along the $x$ axis, then move straight in
the $u$-direction  until reaching $u=u_M$, and finally come along
the $x$-direction to $(u_M,0)$). In particular, the
$d_s^{F_\beta}$-ball of radius $R_0$ is non-compact.

Now, take $D=(-1,1)\subset \R$. As $\bar D$ is compact,  the
intersection of the closed $d^{h}$-balls with $\bar D$ are
compact, that is, {\em the subsets $\bar B_s(p_0,r_1)\cap \bar D$,
$p_0\in D$, $r_1>0$ are always compact}. However, {\em the
analogous subsets $\bar B_s((u_1,p_1), r_1)\cap (\R_u\times\bar
D)$ in $\R_u\times S$ are not compact} for $r_1\geq R_0$, as they
include all the region $\R_u\times \bar D$. Nevertheless, the
intrinsic balls $\bar B_s^{\R_u\times \bar D}((u_1,p_1), r_1)$ are
compact (in agreement with the next lemma), as $\beta$ is bounded
on $\bar D$.
\end{example}

\begin{lem}\label{lcompact} Assume that the intrinsic balls
$\bar B_s^{\bar
D}(p_0,r_0)$, $p_0\in D$, $ r_0>0 $ are compact subsets in $\bar D$.
Then, the intrinsic balls $\bar B_s^{\R_u\times\bar D}((u_1,p_1),
r_1)$ are also compact subsets in $\R_u\times \bar D$.
\end{lem}
\begin{proof}
 It is enough to check that $\bar B_s^{\R_u\times\bar
D}((u_1,p_1), r_1)$ lies in a compact product subset of
$\R_u\times \bar D$.  Let $(u_2,p_2)$ be a point in $\bar
B_s^{\R_u\times\bar D}((u_1,p_1), r_1)$ and let $\gamma^i$,
$i=1,2$, be two curves $\gamma^i\colon[0,1]\to \R_u\times  \bar
D$, $\gamma^i(s)=(u^i(s), x^i(s))$,  $\gamma^1$ from $(u_1,p_1)$
to $(u_2,p_2)$  and  $\gamma^2$ from  $(u_2,p_2)$ to $(u_1,p_1)$,
such that $\int_0^1F_\beta(\dot \gamma^i)\de s< 2r_1 $. Hence, for
each $s\in [0,1]$, we have $d^{\bar D}(p_1,x^1(s))<2r_1$ and
$d^{\bar D}(x^1(s),p_1)<  4r_1 $ (for the latter inequality, it is
enough to consider the arc of $x^1$ from $x^1(s)$ to $p_2$ and
then the curve $x^2$ from $p_2$ to $p_1$). Thus $d^{\bar D}_s(p_1,
x^1(s))< 3  r_1$  for all $s\in [0,1]$. As $\bar B^{\bar D}_s(p_1,
3r_1 )$ is compact, there exist positive constants $A, B, C$,
independent of the curve $x^1$, such that $\beta(x^1(s))\leq A^2$
and $|\omega(\dot x^1(s))|\leq B |\dot x^1(s)|_h\leq C F(\dot
x^1(s))\leq C F_\beta(\dot \gamma^1(s))$, for all $s\in [0,1]$.
Thus, we have \baln
|u_2-u_1|&\leq \int_0^1|\dot u^1|\de s\leq A\int_0^1 \left(\frac{(\dot u^1)^2}{\beta}+
h(\dot x^1, \dot x^1)\right)^{1/2}\de s\\
&\quad  =    A\left(\int_0^1 F_\beta(\dot\gamma^1)\de s-\int_0^1
\omega(\dot x^1)\de s\right)\leq 2 Ar_1(1+C)  =: K,  \ealn hence
the closed ball  $\bar B_s^{\R_u\times\bar D}((u_1,p_1), r_1)$ is
contained in the compact subset $[u_1-K, u_1+K]\times \bar
B_s^{\bar D}(p_1,  3 r_1)$ of $\R_u\times S$.
\end{proof}

\begin{proof}[Proof of Theorem~\ref{appl2}.]
As described  in Subsection~\ref{timeconvexity},   the  existence of
future--pointing timelike ge\-o\-des\-ics  having Lorentzian
length $\ell $ and connecting the point $(p,t_p)$ to the line
$l_q(\tau)=(q,\tau)$ in the region $D\times \R$ contained in
$(L,g_L)$, is equivalent to the existence  of future--pointing
lightlike geodesics connecting the point $(0,p,t_p)$ and the line
$\tilde l_q(\tau)=( \ell ,q,\tau)$ in the region $\R_u\times D\times
\R$ contained  in  $(L_1 , g_{L_1})$ (see \eqref{enne}). In turn,
the latter is equivalent, by Proposition~\ref{fp}, to the
existence of geodesics, with respect to the Randers metric
$F_\beta$, connecting the points $(0,p)$ and $( \ell ,q)$ in
$\R_u\times D$ and then,  from Theorem~\ref{main},
Remark~\ref{rapostilla} and Lemma~\ref{lcompact},  it  is equivalent to the convexity of
$(\R_u\times  \partial D;F_\beta)$.
\end{proof}

\begin{remark}\label{referee2rem3}
The applicability of Theorem \ref{appl2} is determined by
Proposition \ref{Rbetaconvexity2} (and, in particular, Corollary
\ref{betabehaviour}).  For the case of the stationary region $M^a$
of Kerr spacetime, Theorem \ref{appl2} is applicable by choosing
$D$ equal to the domain $D^a_\eps$ in \eqref{eda} for  any $\eps\in (0,m)$ and any small enough 
$|a|\geq 0$  (as $\partial M^a_\eps$ is then
time-convex,  as discussed above Corollary \ref{cKerrlc})  but not to the domain
$D^a_{\eps,\ro}$ in Corollary \ref{cKerrlc} ($M^a_{\eps,\ro}$ is
only light-convex).

The violation of time-convexity (physically interpreted in Remark
\ref{fall}) can be analyzed by considering connecting timelike
geodesics of prescribed length $|v|>0$ between some fixed $p$ and
$l_q$. Recall that the second and third terms in the right hand
side of \eqref{tcon} goes to 0 for large $\ro$, while the first
one remains of the order $-2|y|^2$, according to
\eqref{ehessasympInequality}. So, a geodesic $\gamma$ which
violates time-convexity (in the sense that remains in the closure
of $ M^a_{\eps,\roo}$ but touches  the component  $ S^2(\roo) \times\R$ of its boundary) will
have small
 $|y|$-component.
This corresponds with the component along
 $T S^2(\roo) $ of $\dot\gamma$
and it is  related  to the angular momentum
of $\gamma$. Geodesics with $|y|/|v|$ greater than some  positive constant
will not violate time-convexity for sufficiently large
radius $\ro$.

\end{remark}

\begin{remark}\label{referee2rem4}  Notice also that, in
cosmological models, the metric is typically {\em globally
conformal} to a stationary one (actually, a static one)  and, thus, the techniques are
applicable to this case. For example, in a   FLRW model, the
spacetime is written as a warped product $(I\times S,
g=-dt^2+f(t)^2\pi_S^*g_S)$ where $I\subset \R$ is an interval,
$\pi_S: I\times    S  \rightarrow S, t: I\times   S   \rightarrow I$ are
the natural projection and $(S,g_S)$ is Riemannian manifold. The
conformal metric $g/f^2$ can be written as a product spacetime
$J\times S$, where $J\subset \R$ is another interval determined by
the change of variable $ds=dt/f(t)$. To obtain results on
connecting lightlike geodesics, it is enough to apply the former
ones to the latter product, taking into account that the
$s$-interval covered by the geodesic must be included in $J$. \em
\end{remark}

\subsection{Appendix: revision of the techniques from the SRC
viewpoint}\label{s4.3}

At the beginning of Subsection \ref{s4.1}, we explained how the
result of existence of  at least one  lightlike connecting
geodesic in manifolds without boundary in \cite{cym} could  be
obtained from purely causal grounds. Next, we will consider the
direct implications of causality for manifolds with boundary  on the existence of connecting causal geodesics. 
We will focus on the  case of timelike
geodesics, as for lightlike ones, the question is simpler (recall
 Remark~\ref{rrr} (2)). We will see how direct techniques of
causality plus SRC allow to prove Proposition \ref{propchuleria}
below. Then, we will discuss  the different techniques and
results.

Let us start with preliminary results on causality.  Recall that,
as far as causality is only involved, $C^1$ regularity for the
ambient manifold and  the  domain $D$, and $C^0$ for the
Lorentzian metric will be enough.
\begin{lem}\label{lemmin} A product  spacetime
$(\R_u\times M, g_1=du^2+g_L)$ is causally
simple iff
\begin{enumerate}
\item[(i)] $(M,g_L)$  is causally simple, \item[(ii)]  the
time-separation $d_L$ for $(M,g_L)$ is continuous between points
$w,w'\in M$ with $d_L(w,w')<\infty$  (here,
$d_L(w,w')=\sup_{z\in\mathcal C(w,w';M)} \ell_{g_L}(z)$, where
$\mathcal C(w,w';M)$ is the set of the future-pointing causal
curves $z$ from $w$ to $w'$, and
$\ell_{g_L}(z)=\int_z\sqrt{-g_L(\dot z,\dot z)}$ is the Lorentzian
length),
\item[ (iii)] if $w\leq w'$ and $d_L(w,w')<\infty$
then there exists a future--pointing causal geodesic $\sigma$ from
$w$ to $w'$ with length equal to $d_L(w,w')$.
\end{enumerate}
\end{lem}
 \begin{proof} This is just a particular case of
\cite[Theorem 3.13]{minguzzi}, which is stated for any product
$H\times M$, (apply it to $H=(\R_u,du^2)$).
\end{proof}
Recall that, in the previous result on connectivity, the
finiteness of  the time separation $d_L$  becomes essential to
ensure both, the existence  of connecting causal geodesics and the
continuity of $d_L$.  As the second statement of the following
result shows, in stationary spacetimes the finiteness of the time
separation $d$ of $D\times \R$ (regarded as a spacetime by itself,
i.e.,  $d(w,w')=\sup_{z\in\mathcal C(w,w';D\times
\R)}\ell_{g_L}(z)$) is ensured by the compactness of the closed
symmetrized balls.

\begin{lem}\label{lemmin2} For any stationary domain
$(D\times \R, g_L)$:
\begin{enumerate}
\item   The time-separation $d$ is unbounded on the stationary
lines, i.e.\ for all $w=(p,t_p)\in D\times \R$ and $q\in D$,
$\lim_{\tau\rightarrow\infty} d(w,l_q(\tau))=\infty$. \item  If
the intrinsic balls $\bar B_s^{\bar D}(p_0,r)$, $p_0\in D$, $r>0$,
are compact subsets in $\bar D$, then the time-separation $ d $ is
finite valued on $D\times \R$.
\end{enumerate}
\end{lem}
\begin{proof} (1) Recall that $w$ and $l_q$ can be joined by means
of a future-directed lightlike curve $z(s)=(x(s),t(s)), s\in
[0,1]$ (choose any curve $x$ in $D$  connecting $p$ and $q$, and
choose $t$ so that $z$ becomes lightlike), and that the length of
$l_q$ between $z(1)$  and $l_q(\tau)$ goes to infinity for large
$\tau$.

(2) Assume by contradiction that $d((p,t_p),(q,t_q))=\infty$, and
let $$z_n(s)=(x_n(s),t_n(s)), \quad s\in [0,1]$$ be a sequence of causal
curves in $D$ from $(p,t_p)$ to $(q,t_q)$ with diverging
Lorentzian  lengths. Any point $z_n(s)$ lies in $J^+((p,t_p))\cap
J^-((q,t_q))$ and, so, all the curves $x_n$ lie in the
intersection between the closures of the forward
$F$-ball $B^{\bar D+}(p,t_q-t_p)$ and the backward one $B^{\bar
D-}(q,t_q-t_p)$ (use \cite[Eq. (4.6)]{cys}), and then they are
contained  in $K:=\bar B_s^{\bar D}(p, t_q-t_p+\frac{d^{\bar
D}(q,p)}{2})$ which is  a compact set.
As the Fermat metric $F=\sqrt{h}+\omega$ is positive definite,
there exists some $\eps>0$ and $\eta>0$ such that the $h$-norm of
$\omega$ satisfies $ \|\omega\|_x  <1-\epsilon$ and $\beta(x)\leq \eta^2$,
for all $x\in K$.  Then, as the
arrival times of all the curves $\{z_n\}$ is $t_q$, equation
\eqref{ar1} implies that the $h-$length of all the curves $x_n$ is
bounded.

Then, parameterizing the curves $z_n$ at constant speed gives
$ \ell_n =\!\sqrt{-g_L(\dot z_n, \dot z_n)} = \int_0^1\sqrt{g_L(\dot
z_n(s), \dot z_n(s))}\de s\to \infty$ and, from \eqref{ar1}, we get a
contradiction
\baln \lefteqn{t_q-t_p +(1-\eps)
\int_0^1\sqrt{h(\dot x_n,\dot x_n)}\de s\geq
t_q-t_p  -\int_0^1\omega(\dot x_n)\de s}&\\
&=\int_0^1\sqrt{h(\dot x_n,\dot x_n)+\frac{v_n^2}{\beta(x_n)}}\de s\geq
\int_0^1\sqrt{\frac{ \ell_n^2 }{\beta(x_n)}}\de s\geq \frac{1}{\eta}  \ell_n \to \infty.
\ealn
\end{proof}
As a consequence of the previous two lemmas, we can give the
following relevant particular case of Theorem \ref{appl2} (recall
also Remark \ref{rappl2}(1)) by using strictly causal hypotheses
and proof (including  stationary-to-Randers correspondence, SRC).
\begin{prop}\label{propchuleria} 
Assume that a stationary domain
$(D\times \R, g_L)$ satisfies that $(\R_u\times D\times \R, \,
g_{L_1}=\Pi_1^*du^2 + \Pi^*g_L)$ is causally simple  and that the intrinsic balls $\bar B_s^{\bar
D}(p_0,r)$, $p_0\in D$, $r>0$,
are compact subsets in $\bar D$.  Then any
point $w = (  p, t_p ) \in D\times \R$ and any line $l_q(\tau) =
(q, \tau) \in D\times \R$, $\tau \in \R$,
 can be joined by a future--pointing timelike geodesic $z (s) = (
x (s), t (s)) \in D \times \R$, $s\in I=[0,1]$ with Lorentzian
length $ \ell $, such that $x$  minimizes the arrival time $t(1)$
among all the future-pointing causal curves from $p$ to $q$ of the
same length $ \ell $.
\end{prop}
\begin{proof}
Notice that, as $d$ is finite-valued from Lemma \ref{lemmin2}(2),
the assertion (iii) of Lemma \ref{lemmin} implies that
Avez-Seifert property holds, i.e. each two  points $w,w'\in
D\times \R$ which are  strictly causally related ($w< w'$), can be
connected by means of a causal geodesic of length  $d(w,w')\in (0,+\infty)$. Let
$\tau_0$ be the infimum of the $\tau\geq t_p$ such that $w\leq
(q,\tau)$. Clearly, $d(w,(q,\tau_0))=0$ (by the assertion (i) of
Lemma \ref{lemmin}, $w\leq (q,\tau_0)$ too) and  the property in
Lemma \ref{lemmin2}(1) plus the continuity of $d$ (assertion (ii)
in Lemma \ref{lemmin}) imply that the  time-separation between $w$ and some 
point of $l_q$ is $ \ell $. Thus, Avez-Seifert property yields the
result.
\end{proof}
As a consequence, the considerations about the generality of our
results in the lightlike case in comparison with the the previous
ones (Remark \ref{rrr}), can be extended to the case of Theorem
\ref{appl2}. Summing up, our conclusions about the obtained
results and required techniques are the following ones.
\begin{enumerate}
\item In the purely causal proof of Proposition
\ref{propchuleria}, SRC has been used to show the finiteness of
$d$ (Lemma \ref{lemmin2}), which is crucial for the Avez-Seifert
property.
 In comparison with the results in the previous subsection, the
limitations of this causal result  \ref{propchuleria} are:

(a) it does not explain when $(\R_u\times D\times \R, \,
g_{L_1}=\Pi_1^*du^2 + \Pi^*g_L)$ is causally simple in terms of
the stationary spacetime $(D\times\R,g_L)$,   and

(b) it does not allow a   result on the multiplicity of the connecting timelike geodesics. 

\item The limitation (a) is again remedied by SRC  (Theorem~\ref{appl1} applied to the stationary spacetime
$(\R_u\times D\times\R,g_{L_1})$).  For the limitation (b),  an
infinite-dimensional variational approach is
required\footnote{However, one could still find a result of
multiplicity in purely causal terms by using timelike homotopy
classes as in \cite{SaProcAMS}, which allows some sharp
conclusions on the behavior of the geodesics.}. In fact,  the
result  in Theorem \ref{main} (plus Remark
\ref{rapostilla}) was claimed in the proof of our main  Theorem
\ref{appl2}.  The formulation of the hypothesis of convexity with
the function $\phi$ allows to connect the geometric
interpretations with the variational approach, which uses a
functional with a penalization  term   constructed from $\phi$  (see \cite[Section 4]{bcgs}).

\item At any case, the statement and proof of our main result
(Theorem \ref{appl2} plus Remark \ref{rappl2}) requires both,
variational results  and  causal interpretations, including SRC.
This allows to obtain optimal analytical hypotheses. In fact, we
use only the overall hypothesis of the compactness of closed
symmetrized balls, which becomes natural even in the Riemannian
case (as explained in Remark \ref{rrr}(1)), and has a role clearly
revealed in Lemma \ref{lemmin2}(2)). Up to this ambient
hypothesis,  our   condition for the problem of causal connectedness
is both, necessary and sufficient.

 \item As a consequence, our
results improve all the previous references on the topic.
Essentially these references were obtained by using variational
methods, and achieved sufficient conditions for causal
connectedness. In general, typical analytic  conditions imply the
conditions in
 Proposition \ref{Rbetaconvexity2} which characterize  (together with Theorem~\ref{peq2} and Remark~\ref{timeequiv}) the time-convexity of $\partial D\times\R$. 

 In particular, Theorem~\ref{appl2} improves the
results in \cite[Sect. 4.3]{cym}, by dealing with manifolds with
boundary, and those in  \cite[Th. 1.6, 1.7]{gm} by giving the full
characterization of causal connectedness with natural geometric
interpretations. The results in \cite{bg} obtained by means
of a different approach based on a relation between geodesics and
Lagrangian systems, are also improved. In this reference,  the
existence of timelike geodesics between a point and a line  of the
boundary was proved, under time--convexity, only for a suitable
range of  values for the Lorentzian length $\ell$,  depending on
the metric coefficients.
\end{enumerate}
\begin{remark}
For further developments, the following possibilities are pointed
out. First, our approach based on SRC is potentially useful also
for other variational problems, such as periodic trajectories or
trajectories critical for another (time-independent) functionals,
see \cite{bs} and references therein.

However, we emphasize that only causal geodesics are studied by
using SRC. The problems which include  spacelike geodesics, as geodesic
connectedness, must introduce also other subtle techniques. It is
worth pointing out, in relation with the cases obtained here:

(a) domains such as $M^{a}_\epsilon, \epsilon>0$  in Kerr spacetime (see Subsection~\ref{kerrshell})  not only are
  not   space-convex but also  are not
geodesically connected \cite[Corollary 3]{FlSa},

(b) the full outer region of Schwarzschild spacetime, as well as
the outer region  of slow Kerr one (outside the black hole, which
include the ergosphere and, so, it is not fully stationary for
$a\neq 0$) is geodesically connected \cite{FlSa2}; the proof uses
different topological arguments introduced in
\cite{FlSaJournaldifferentialequations2002}, and

(c) for general standard  stationary spacetimes, geodesic
connectedness has been studied by a combination of variational and
causal methods in \cite{CFSadvances}  (see also \cite{caponiogelogra}, for the case of domains with boundary).

\smallskip

\noindent Even though such problems of geodesic connectedness do
not have a simple physical interpretation as those of {\em causal}
geodesic connectedness in this paper, they constitute and
excellent arena to study critical points curves for indefinite
functionals, with broad possibilities of applications.
\end{remark}

\section{Conclusions}

The problem of visibility of stellar objects (existence and
multiplicity of causal geodesics  connecting points and  world-lines)
under physically reasonable properties has been analyzed, being
the following ingredients relevant in its solution:

\begin{itemize}

\item Relativistic Fermat's principle, as known from the works by
Kovner \cite{Kov} and Perlick \cite{Per}, asserts that if a
connecting causal curve from a point to a observer (world-line) with a critical arrival time for the observer's proper time
exists, then it is a lightlike geodesic. However, in order to
ensure the existence of such a geodesic, a mixture of variational
techniques (critical point theory), topological elements
(Ljusternik-Schnirelmann theory) and geometrical equivalences
(stationary-to-Randers correspondence) in the framework of
Causality theory, has been used. In particular:

(i) The existence of connecting causal geodesics is  characterized
in terms of the geodesic connectedness of   the  associated Finsler
manifolds of Randers type. Therefore, the notion of convexity (for
domains of the spacetime) is analyzed
carefully.

(ii) The multiplicity of connecting geodesics (lensing) appears naturally either  due to curvature  (the geodesics of the associated Fermat metric present conjugate points, a well understood property in Finsler Geometry)
or to global
topological properties  (non-contractibility of the manifold).

\item We have considered only  stationary spacetimes   (or {\em strictly} stationary in the nomenclature of some references, as we are assuming that the causal character of the Killing vector field must remain timelike on all the spacetime).
Nevertheless, the applications to asymptotically flat spacetimes
make the results applicable in more general situations  which
include black holes, and the conformal invariance of most of the
techniques make them applicable even to cosmological models.

\item It is natural to consider the case of connecting geodesics
that  are confined in a spherical shell   of the spacetime and, so, large balls
in asymptotically flat spacetimes have been  especially studied.
The obtained results show under which circumstances  our intuition
on known spacetimes as Kerr's one can be transplanted to arbitrary
asymptotically flat spacetimes  (recall Proposition~\ref{positivemass}).

\item The stationary to Randers correspondence allows a better understanding of the notion of asymptotic flatness, connecting it with the asymptotic behaviour of a Finsler manifold and interpreting the role of the cohomology of the shift $\omega$ as 
 a significant geometric object in that notion.

\item The results include not only lightlike geodesics but also
timelike ones. For these  geodesics,  it is assumed that they arrive
in a prescribed lifetime. This becomes natural from both
viewpoints, the mathematical one (otherwise, the results on
multiplicity would become trivial) and the physical one (the
particles might disintegrate).

\item The revision  of the previous techniques in the literature
 shows  the limitations of each one,   as well as the optimality of the obtained
results.
\end{itemize}

\subsection*{Acknowledgment}
We would like to thank R. Bartolo for several discussions on a
preliminary version of this work.

\end{document}